\documentclass[fleqn,a4paper,9pt]{article}
\usepackage{amsmath}
\usepackage{mathtools}
\usepackage{mathptmx}
\usepackage{comment}
\usepackage{tabularx}
\usepackage{listings}
 \usepackage{amsthm} 
\usepackage{amssymb}

\makeatletter
\let\@@pmod\pmod
\DeclareRobustCommand{\pmod}{\@ifstar\@pmods\@@pmod}
\def\@pmods#1{\mkern4mu({\operator@font mod}\mkern 6mu#1)}
\makeatother

\newcommand{\tpmod}[1]{\mkern 8mu({\operator@font mod}\mkern 6mu#1)}

\usepackage{atbegshi}
\AtBeginDocument{\AtBeginShipoutNext{\AtBeginShipoutDiscard}}

\author{Leon Fairbanks}
\begin{document}
\begin{flushleft}
\title{{\textbf{Sums of Powers of Sine and Generalized Bernoulli Polynomials}}}
\maketitle
\begin{abstract}
We produce formulas for 
$$\sum_{j=1}^{2^{n-2}}\frac{1}{\sin^s\left(\frac{(2j-1)\pi}{2^n}\right)}$$\notag in terms of Generalized Bernoulli and Euler polynomials
and use one of the formulas to produce a nice integral representation of the Riemann zeta function.
\end{abstract}
\tableofcontents

\newtheorem{theorem}{Theorem}[section]
\newtheorem{lemma}[theorem]{Lemma}
\newtheorem{proposition}[theorem]{Proposition}
\newtheorem{corollary}[theorem]{Corollary}
\newtheorem{definition}[theorem]{Definition}

\section{Introduction}
We derive some formulas for sums of powers of the sine function and make explict the relationship to the Riemann zeta function. The following proposition lists some formulas of the type we consider here.

\begin{proposition}
Let
\begin{align}\notag
S(s,n)&=\sum_{j=1}^{2^{n-2}}\frac{1}{\sin^s\left(\frac{(2j-1)\pi}{2^n}\right)}\\\notag
\end{align}
Then
\begin{align}\notag
S(2,n)&=\frac{1}{2}2^{2n-2}\\\notag
S(3,n)&=\sum_{j=1}^{2^{n-2}}\left(-2j^2+2(2^{n-1}+1)j-2^{n-1}\right)\frac{1}{\sin[\frac{(2j-1)\pi}{2^n}]}\\\notag
S(4,n)&=\frac{1}{6}\left(2^{4n-4}+2(2^{2n-2})\right)\\\notag
S(5,n)&=\frac{1}{3}\sum_{j=1}^{2^{n-2}}\biggl(2 j^4-4(2^{n-1}+1)j^3+2(3( 2^{n-1})-1)j^2+2(2^{3n-3}+2^{n-1}+2)j\\\notag
&-(2^{3n-3}+2(2^{n-1}))\biggr)\frac{1}{\sin[\frac{(2j-1)\pi}{2^n}]}\\\notag
S(6,n)&=\frac{1}{30}(2(2^{6n-6})+5(2^{4n-4})+8(2^{2n-2}))\\\notag
S(7,n)&=\frac{1}{45}\sum_{j=1}^{2^{n-2}}\biggl(  -4 j^6+12(2^{n-1}+1)j^5-10(3(2^{n-1})-2)j^4-20(2^{3n-3}+2(2^{n-1})+3)j^3 \\\notag
&+2(15(2^{3n-3})+45(2^{n-1})-8)j^2 +4( 3(2^{5n-5})+5(2^{3n-3})+4(2^{n-1})+12 )j\\\notag
&-3( 2(2^{5n-5})+5(2^{3n-3})+8(2^{n-1}) )\biggr)\frac{1}{\sin[\frac{(2j-1)\pi}{2^n}]}\\\notag
S(8,n)&=\frac{1}{630}\left(17(2^{8n-8})+56(2^{6n-6})+98(2^{4n-4})+144(2^{2n-2})\right)\\\notag
\end{align}
\end{proposition}
\begin{proof}
See ref. [1],[2].
\end{proof}
A connection with the Riemann zeta function:

\begin{align}
\zeta(m)&=\lim_{n\to\infty}\left(\frac{2^m\pi^m}{2^{m}-1}\right) \sum _{i=1}^{2^{n-2}}\frac{1}{ \left(2^n\sin\left(\frac{\left(2 i-1\right)\pi}{2^n}\right)\right)^m} \\\notag
&=\lim_{n\to\infty}\left(\frac{\pi^m}{2^{m}-1}\right) \left(\frac{1}{2^{mn-m}}\right)\sum _{i=1}^{2^{n-2}}\frac{1}{ \left(\sin\left(\frac{\left(2 i-1\right)\pi}{2^n}\right)\right)^m} \\\notag
&=\lim_{n\to\infty}\left(\frac{\pi^m}{2^{m}-1}\right)\left(\frac{1}{2^{n-1}}\right) \left(\frac{1}{2^{(m-1)(n-1)}}\right)S(m,n) \\\notag
\end{align}

The two main results are:\\
Theorem \ref{last}
Let $E_j(x)$ be the j-th Euler polynomial and $j\in\mathbb{N}$ then
\begin{align}
S(2m+1,n)&=(-1)^m\frac{2^{2mn}}{(2m)!}\sum_{j=1}^{2^{n-2}}\sum_{k=0}^{m-1}2^{(n-1)(-2k-1)}\binom{2m-1}{2k}B_{2k}^{(2m)}(m)\\\notag
&\Biggl(2^{n-1}E_{2m-2k}\left(\frac{j}{2^{n-1}}\right)-mE_{2m-2k-1}\left(\frac{j}{2^{n-1}}\right)\Biggr)\frac{1}{\sin[\frac{(2j-1)\pi}{2^n}]}\\\notag
\end{align}
Theorem \ref {theorem1}
 Let $B(j)$ be the $j$-th Bernoulli number and let $B_n^{(\alpha)}(x)$ be the generalized Bernoulli polynomial of degree $n$ in $x$.
\begin{align}\notag
S(2 k,n)&=\sum_{j=1}^k(-1)^{2j+k+1}\frac{2^{2j(n-1)+2k-1}(2^{2j}-1)}{(2(k-j))!(2j)!}B_{2(k-j)}^{(2k)}(k) B(2 j)\\\notag
&=\sum_{j=1}^k(-1)^{j+k}\frac{2^{2j(n-2)+2k}(2^{2j}-1)}{\pi^{2j}(2(k-j))!}B_{2(k-j)}^{(2k)}(k) \zeta(2 j)\\\label{sumsine}
\end{align}

The section on generalized Bernoulli polynomials sets up the machinery for the proof of the theorems.
The integral for the Riemann zeta function is
\begin{align}
\zeta(2j+1)&=(-1)^j\frac{2^{2j}\pi^{2j+1}}{(2j)!(2^{2j+1}-1)}\int_{0}^{1/2}E_{2j}(x)\csc (\pi  x)dx\\\notag
\end{align}
\section{Generalized Bernoulli Polynomials}
The generalized Bernoulli polynomials are defined via the generating function
$$\left(\frac{t}{e^t-1}\right)^\alpha e^{xt}=\sum_{n=0}^\infty B_n^{(\alpha)}(x)\frac{t^n}{n!}$$
The Euler polynomials  are defined
$$\left(\frac{2}{e^t+1}\right) e^{xt}=\sum_{n=0}^\infty E_n(x)\frac{t^n}{n!}$$
We list some known equivalences for future reference.
\begin{proposition}
 Let $E_j(x)$ be the j-th Euler polynomial, $E(k)$ the $k$-th Euler numberr, $B_n^{(\alpha)}(x)$ the generalized Bernoulli polynomial of degree $n$ in $x$, $B(j)$ the $j$-th Bernoulli number.
\begin{align}
B(m)&=\sum_{k=0}^m\frac{1}{k+1}\sum_{j=0}^k\binom{k}{j}(-1)^jj^m\\
2^{2n}B_{2n}^{(2x)}(x)&=\lim_{t\to 0}\frac{  \mathrm{  d^{2n}  }   }{\mathrm{d}t^{2n}  }        \Bigl(\frac{t}{\sinh(t)}\Bigr)^{2x}\\
(-1)^{j+1}2^{2j}(2^{2j}-1)B(2j)&=\frac{  \mathrm{  d^{2j}  }   }{\mathrm{d}t^{2j}  }        \Bigl(\tan(t))\Bigr)\vert_{t=0}\\
(-1)^{j+1}\frac{2^{2j}}{(2j)!}B(2j)&=\frac{2}{\pi^{2j}}\zeta(2j)\\
E(k)&=2^kE_k(\frac {1} {2}) \\
E(2k+1)&=0\\
E_0(x)&=1\\
B_{0}^{(\alpha)}(x)&=1\\
E_{2k}(0)&=0, \rm{\  for\  }k>0\\
E_n(0)&=\frac{2}{n+1}\left(1-2^{n+1}\right)B_{n+1}(0)\\
E_j(x)&=\sum_{k=0}^j\frac{1}{2^k}\sum_{l=0}^j(-1)^l \binom{k}{l}(x+l)^j\\
\end{align}
\end{proposition}
\begin{proof}
\end{proof}
\begin{proposition}\label{xplusy}
\begin{align}
B_k(x+y)&=\sum_{i=0}^k\binom{k}{i}B_i(x)y^{k-i}\\\notag
E_k(x+y)&=\sum_{i=0}^k\binom{k}{i}E_i(x)y^{k-i}\\\notag
\end{align}
\end{proposition}
\begin{proof}
From the above,  $B_{n}^{(\alpha)}(x+y)$ corresponds to generating function 
\begin{align}
\left(\frac{t}{e^t-1}\right)^\alpha e^{(x+y)t}&=\sum_{n=0}^\infty B_n^{(\alpha)}(x+y)\frac{t^n}{n!}\\\notag
e^{yt}\left(\frac{t}{e^t-1}\right)^\alpha e^{xt}&=\sum_{n=0}^\infty\sum_{i=0}^n B_i^{(\alpha)}(x)y^{n-i}\frac{t^n}{i!(n-i)!}\\\notag
&=\sum_{n=0}^\infty\sum_{i=0}^n \binom{n}{i}B_i^{(\alpha)}(x)y^{n-i}\frac{t^n}{n!}\\\notag
\end{align}
The same argument applies to Euler polynomial generating function.
\end{proof}
\begin{proposition}\label{addAlpha}
For $i>0$, $m>0$
\begin{align}
B_{n}^{(\alpha)}(\alpha -x)&=(-1)^n B_n^{(\alpha)}(x)\\\notag
\end{align}
\end{proposition}
\begin{proof}
From the above,  $B_{n}^{(\alpha)}(\alpha -x)$ corresponds to generating function 
\begin{align}
\left(\frac{t}{e^t-1}\right)^\alpha e^{(\alpha-x)t}&=\sum_{n=0}^\infty B_n^{(\alpha)}(\alpha-x)\frac{t^n}{n!}\\\notag
&=\left(\frac{t e^t}{e^t-1}\right)^\alpha e^{-xt}\\\notag
&=\left(\frac{-t}{e^{-t}-1}\right)^\alpha e^{x(-t)}\\\notag
&=\sum_{n=0}^\infty (-1)^n B_n^{(\alpha)}(x)\frac{t^n}{n!}\\\notag
\end{align}

\end{proof}
\begin{proposition}\label{alphaMinus}
\begin{align}
B_n^{(\alpha-1)}(x)&=\frac{1}{n+1}\left(B_{n+1}^{(\alpha)}(x+1)-B_{n+1}^{(\alpha)}(x)\right)\\\notag
&=\frac{1}{n+1}\sum_{k=0}^n\binom{n+1}{k}B_k^{(\alpha)}(x)\\\notag
\end{align}
\end{proposition}
\begin{proof}
Consider the generating function
\begin{align}
&\left(\frac{t}{e^t-1}\right)^\alpha e^{(x+1)t}-\left(\frac{t}{e^t-1}\right)^\alpha e^{xt}\\\notag
&=e^t\left(\frac{t}{e^t-1}\right)^\alpha e^{xt}-\left(\frac{t}{e^t-1}\right)^\alpha e^{xt}\\\notag
&=(e^t-1)\left(\frac{t}{e^t-1}\right)^\alpha e^{xt}\\\notag
&=t\frac{e^t-1}{t}\left(\frac{t}{e^t-1}\right)^\alpha e^{xt}\\\notag
&=t\left(\frac{t}{e^t-1}\right)^{\alpha-1} e^{xt}\\\notag
\end{align}
If we look at the coefficient of  $t^{n+1}$ in the first line, we get $$B_{n+1}^{(\alpha)}(x+1)-B_{n+1}^{(\alpha)}(x)$$
If we look at the coefficient of  $t^{n+1}$ in the second line, we get $$\sum_{k=0}^{n+1}\binom{n+1}{k}B_k^{(\alpha)}(x)-B_{n+1}^{(\alpha)}(x)$$
If we look at the coefficient of $t^{n+1}$ in the last line, we get $$(n+1)B_n^{(\alpha-1)}(x)$$
\end{proof}
\begin{proposition} \label{oddZero}
  For $i>0$, $m>0$ ($i\in\mathbb{Z}$, $m\in\mathbb{R}$).
\begin{align}
B_{2i+1}^{(m)}(\frac{m}{2})&=0\\\notag
\end{align}
\end{proposition}
\begin{proof}
The generating function is even:
\begin{align}
\left(\frac{-t}{e^{-t}-1}\right)^{m} e^{-\frac{m}{2}t}&=\left(\frac{te^t}{e^{t}-1}\right)^{m} e^{-\frac{m}{2}t}\\\notag
&=\left(\frac{t}{e^{t}-1}\right)^{m} e^{mt-\frac{m}{2}t}\\\notag
\end{align}
\end{proof}

\begin{lemma}\label{bidentity}
\begin{align}
\binom{n+k}{n-r}\binom{n+r+1}{n+k+1}&=\binom{n+k}{k}\binom{n}{r}\binom{n+r+1}{r-k}\frac{r!k!}{(r+k)!}\\\notag
\end{align}
\end{lemma}
\begin{proof}
Cancelling $\binom{n+r+1}{n+k+1}$ from both sides, this is equivalent to
\begin{align}
\binom{n+k}{r+k}\binom{r+k}{r}&=\binom{n+k}{n}\binom{n}{r}\\\notag
\end{align}
Using the fact that
\begin{align}
\binom{n}{k}\binom{k}{j}&=\binom{n}{j}\binom{n-j}{k-j}\\\notag
\end{align}
on both sides, we see equality.
\end{proof}
\begin{lemma}\label{sS}
Let $s(i,j)$ and $S(i,j)$ be the Stirling numbers of the first and second kind, respectively.
\begin{align}
s(n,k)&=(-1)^{n-k}\sum_{j=n}^{2n-k}(-1)^{j-k}\binom{j-1}{k-1}\binom{2n-k}{j}S(j-k,j-n)\\\notag
\end{align}
\end{lemma}
\begin{proof}
See ref. [4]
\end{proof}
\begin{proposition}\label{BigB}
Let $S(i,j)$ be the Stirling number of the second kind.
\begin{align}
B_n^{(\alpha)}(x)&=\sum_{k=0}^n (-1)^k\binom{\alpha+k-1}{k}\sum_{l=0}^{n-k}\binom{n}{l+k}x^{n-k-l}\\\notag
&\binom{\alpha+k+l}{l}\frac{(l+k)!k!}{(l+2k)!}S(l+2k,k)\\\notag
\end{align}
\end{proposition}
\begin{proof}
See ref.[3].
\end{proof}
\begin{proposition}\label{prodOfZeros}
For $1<m$
\begin{align}
B_{m-1}^{(m)}(x)&=\prod_{i=1}^{m-1}(x-i)\\\notag
\end{align}
\end{proposition}
\begin{proof}
From Prop. \ref{BigB}, letting $r=l+k$,
\begin{align}
B_n^{(n+1)}(x)&=\sum_{k=0}^n (-1)^k\binom{n+k}{k}\sum_{l=0}^{n-k}\binom{n}{l+k}x^{n-k-l}\\\notag
&\binom{n+k+l+1}{l}\frac{(l+k)!k!}{(l+2k)!}S(l+2k,k)\\\notag
&=\sum_{r=0}^n\sum_{k=0}^r(-1)^k\binom{n+k}{k}\binom{n}{r}\binom{n+r+1}{r-k}\frac{r!k!}{(r+k)!}S(r+k,k)x^{n-r}\\\notag
\end{align}
while
\begin{align}
\prod_{i=1}^{n}(x-i)&=\sum_{i=0}^n s(n+1,n-i+1)x^{n-i}\\\notag
\end{align}
where $s(i,j)$ is the Stirling number  of the first kind.
From Lemma \ref{sS},
\begin{align}
s(n+1,n-r+1)&=\sum_{k=0}^{r}(-1)^{k+r}\binom{n+k}{n-r}\binom{n+r+1}{n+k+1}S(k+r,k)\\\notag
\end{align}
The proposition then follows from Lemma \ref{bidentity}
\end{proof}
\begin{proposition}
\begin{align}
B_{n-k}^{(\alpha-k)}(x)=\frac{n!}{k!}\sum_{i=0}^k(-1)^{k-i} \binom{k}{i}B_n^{(\alpha)}(x+i)\\\notag
\end{align}
\end{proposition}
\begin{proof}
Consider the generating function
\begin{align}
\left(\frac{t}{e^t-1}\right)^\alpha\left(e^t-1\right)^k e^{xt}&=\left(\frac{t}{e^t-1}\right)^\alpha\left(\sum_{i=0}^k (-1)^{k-i}\binom{k}{i}e^{it}\right)e^{xt}\\\notag
&=\left(\frac{t}{e^t-1}\right)^\alpha\left(\sum_{i=0}^k(-1)^{k-i} \binom{k}{i}e^{(x+i)t}\right)\\\notag
&=\sum_{i=0}^k(-1)^{k-i} \binom{k}{i}\left(\frac{t}{e^t-1}\right)^\alpha e^{(x+i)t}\\\notag
&=\sum_{j=0}^\infty\sum_{i=0}^k(-1)^{k-i} \binom{k}{i}B_j^{(\alpha)}(x+i)\left(\frac{t^j}{j!}\right)\\\notag
\end{align}
while
\begin{align}
\left(\frac{t}{e^t-1}\right)^\alpha\left(e^t-1\right)^k e^{xt}&=t^k\left(\frac{t}{e^t-1}\right)^{\alpha-k}e^{xt}\\\notag
&=\sum_{j=0}^\infty B_j^{(\alpha-k)}(x)\left(\frac{t^{j+k}}{j!}\right)\\\notag
&=\sum_{j=0}^\infty\frac{(j+k)!}{j!} B_j^{(\alpha-k)}(x)\left(\frac{t^{j+k}}{(j+k)!}\right)\\\notag
&=\sum_{j=k}^\infty\frac{j!}{(j-k)!} B_{j-k}^{(\alpha-k)}(x)\left(\frac{t^j}{j!}\right)\\\notag
\end{align}
which shows that
\begin{align}
B_{n-k}^{(\alpha-k)}(x)=\frac{n!}{k!}\sum_{i=0}^k(-1)^{k-i} \binom{k}{i}B_n^{(\alpha)}(x+i)\\\notag
\end{align}
\end{proof}

\begin{corollary}
For $m>1,\text{  } 0\le j<m-1$
\begin{align}
\sum_{k=0}^{m-1}\binom{m-1}{k}B_k^{(m)}(j)&=0\\\notag
\end{align}
\end{corollary}
\begin{proof}
Note $$\sum_{k=0}^{m-1}\binom{m-1}{k}B_k^{(m)}(x)$$ corresponds to the $m-1$-st derivative of
$$e^t\left(\frac{t}{e^t-1}\right)^{(m)} e^{xt}=\left(\frac{t}{e^t-1}\right)^{(m)} e^{(x+1)t}$$
which corresponds to $B_{m-1}^{(m)}(x+1)$, which, by the Proposition \ref{prodOfZeros}, equals 0 for $x=0,\ldots,m-2$.
\end{proof}
\begin{proposition}\label{derivB}
$$B_{n-1}^{(\alpha)}(x)=\frac{1}{n}\frac{ \mathrm{  d  }   }{\mathrm{d}x }B_n^{(\alpha)}(x)$$
\end{proposition}
\begin{proof}

$$\left(\frac{t}{e^t-1}\right)^\alpha e^{xt}=\sum_{n=0}^\infty B_n^{(\alpha)}(x)\frac{t^n}{n!}$$
Differentiating the generating function
\begin{align}
\frac{ \mathrm{  d  }   }{\mathrm{d}x }\left(\left(\frac{t}{e^t-1}\right)^\alpha e^{xt}\right)&=t\left(\frac{t}{e^t-1}\right)^\alpha e^{xt}\\\notag
&=\sum_{n=0}^\infty B_n^{(\alpha)}(x)\frac{t^{n+1}}{n!}\\\notag
&=\sum_{n=0}^\infty \frac{ \mathrm{  d  }   }{\mathrm{d}x }B_n^{(\alpha)}(x)\frac{t^{n}}{n!}\\\notag
\end{align}
$$\implies \frac{1}{n}\frac{ \mathrm{  d  }   }{\mathrm{d}x }B_n^{(\alpha)}(x)=B_{n-1}^{(\alpha)}(x)$$
\end{proof}
\begin{corollary}\label{Dprod}
\begin{align}
B_{m-k}^{(m)}(x)&=\frac{(m-k)!}{(m-1)!}\frac{  \mathrm{  d^{k-1}  }   }{\mathrm{d}x^{k-1}  }  \prod_{i=1}^{m-1}(x-i)\\\notag
\end{align}
\end{corollary}
\begin{proof}
Previous corollary applied to Proposition \ref{prodOfZeros} .
\end{proof}
For reference:
\begin{proposition}
\begin{align}\notag
\frac{ \mathrm{  d  }   }{\mathrm{d}t }&\left(\left(\frac{t}{e^t-1}\right)^{\alpha} e^{xt}\right)
&=\alpha\left(\frac{t}{e^t-1}\right)^{\alpha}\frac{1}{t}e^{xt}-\alpha\left(\frac{t}{e^t-1}\right)^{\alpha+1}\frac{1}{t}e^{xt}  +(x-\alpha)\left(\frac{t}{e^t-1}\right)^{\alpha}e^{xt}\\\notag
\end{align}
\end{proposition}
\begin{proof}
\begin{align}\notag
\frac{ \mathrm{  d  }   }{\mathrm{d}t }&\left(\left(\frac{t}{e^t-1}\right)^{\alpha} e^{xt}\right)=\alpha\left(\frac{t}{e^t-1}\right)^{\alpha-1}\left(\frac{e^t-1-t e^t}{\left(e^t-1\right)^2}\right)e^{xt}+x\left(\frac{t}{e^t-1}\right)^{\alpha}e^{xt}\\\notag
&=\alpha\left(\frac{t}{e^t-1}\right)^{\alpha-1}\left(\frac{e^t-1-t e^t+t-t}{\left(e^t-1\right)^2}\right)e^{xt}+x\left(\frac{t}{e^t-1}\right)^{\alpha}e^{xt}\\\notag
&=\alpha\left(\frac{t}{e^t-1}\right)^{\alpha-1}\left(    \left(\frac{1}{e^t-1}-\frac{t}{(e^t-1)^2}\right)e^{xt}-\left(\frac{t}{e^t-1}\right)e^{xt}        \right) +x\left(\frac{t}{e^t-1}\right)^{\alpha}e^{xt}\\\notag
&=\alpha\left(\frac{t}{e^t-1}\right)^{\alpha}\frac{1}{t}e^{xt}-\alpha\left(\frac{t}{e^t-1}\right)^{\alpha+1}\frac{1}{t}e^{xt}  +(x-\alpha)\left(\frac{t}{e^t-1}\right)^{\alpha}e^{xt}\\\notag
\end{align}
\end{proof}
\begin{lemma}\label{reduced}
For $m, j\ge 0$
\begin{align}\notag
&2(m-j-1)B_{2j+1}^{(2m-1)}(m)-(m-1)(2j+1)B_{2j}^{(2m-1)}(m)=0\\\notag
&(2m-2j-1)B_{2j}^{(2m-1)}(m)-(2m-1)B_{2j}^{(2m)}(m)-(2j)(m-1)B_{2j-1}^{(2m-1)}(m)=0\\\notag
&(2m-2j-1)B_{2j+1}^{(2m)}(m+k)-2mB_{2j+1}^{(2m+1)}(m+k)-(m-k)(2j+1) B_{2j}^{(2m)}(m+k)=0\\\notag
\end{align}
\end{lemma}
\begin{proof}
From the previous proposition,
\begin{align}\notag
\frac{ \mathrm{  d  }   }{\mathrm{d}t }&\left(\left(\frac{t}{e^t-1}\right)^{2m-1} e^{mt}\right)\\\notag
&=(2m-1)\left(\frac{t}{e^t-1}\right)^{2m-1}\frac{1}{t}e^{mt}-(2m-1)\left(\frac{t}{e^t-1}\right)^{2m}\frac{1}{t}e^{mt}  -(m-1)\left(\frac{t}{e^t-1}\right)^{2m-1}e^{mt}\\\notag
\end{align}
But now, thinking of the generator function as a Taylor series in $t$, and taking the derivative wrt $t$, and then matching up the coefficients of $t^{2j}$, we get
\begin{align}\notag
\frac{2m-1}{2j+1}B_{2j+1}^{(2m-1)}(m)-\frac{2m-1}{2j+1}B_{2j+1}^{(2m)}(m)-(m-1)B_{2j}^{(2m-1)}(m)&=B_{2j+1}^{(2m-1)}(m)\\\notag
\end{align}
By Prop. \ref{oddZero}, the second term on the left is zero, yielding,
\begin {align}\notag
2(m-j-1))B_{2j+1}^{(2m-1)}(m)-(2j+1)(m-1)B_{2j}^{(2m-1)}(m)&=0\\\notag
\end {align}
Note, if we look at the coefficients of $t^{2j-1}$ we get
\begin{align}\notag
&\frac{2m-1}{2j}B_{2j}^{(2m-1)}(m)-\frac{2m-1}{2j}B_{2j}^{(2m)}(m)-(m-1)B_{2j-1}^{(2m-1)}(m)=B_{2j}^{(2m-1)}(m)\\\notag
\end{align}
or,
\begin{align}\notag
&(2m-2j-1)B_{2j}^{(2m-1)}(m)-(2m-1)B_{2j}^{(2m)}(m)-(2j)(m-1)B_{2j-1}^{(2m-1)}(m)=0\\\notag
\end{align}
Also, from the previous proposition,
\begin{align}\notag
\frac{ \mathrm{  d  }   }{\mathrm{d}t }&\left(\left(\frac{t}{e^t-1}\right)^{2m} e^{(m+k)t}\right)\\\notag
&=2m\left(\frac{t}{e^t-1}\right)^{2m}\frac{1}{t}e^{(m+k)t}-2m\left(\frac{t}{e^t-1}\right)^{2m+1}\frac{1}{t}e^{(m+k)t} -(m-k)\left(\frac{t}{e^t-1}\right)^{2m}e^{(m+k)t}\\\notag
\end{align}
Again, thinking of the generator function as a Taylor series in $t$, and taking the derivative wrt $t$, and then matching up the coefficients of $t^{2j}$, we get
\begin{align}\notag
&\frac{2m}{2j+1}B_{2j+1}^{(2m)}(m+k)-\frac{2m}{2j+1}B_{2j+1}^{(2m+1)}(m+k)-(m-k) B_{2j}^{(2m)}(m+k)=B_{2j+1}^{(2m)}(m+k)\\\notag
&2mB_{2j+1}^{(2m)}(m+k)-2mB_{2j+1}^{(2m+1)}(m+k)-(m-k)(2j+1) B_{2j}^{(2m)}(m+k)=(2j+1)B_{2j+1}^{(2m)}(m+k)\\\notag
&(2m-2j-1)B_{2j+1}^{(2m)}(m+k)-2mB_{2j+1}^{(2m+1)}(m+k)-(m-k)(2j+1) B_{2j}^{(2m)}(m+k)=0\\\notag
\end{align}
\end{proof}
\begin{lemma}\label{coeffx}
For $m>k\ge 0$,
\begin{align}
\sum_{i=k+1}^m \binom{2m-1}{2i-1}\biggl(\binom{2i}{2k}-m\binom{2i-1}{2k}\biggr)B_{2m-2i}^{(2m)}(m)=0\\\notag
\end{align}
\end{lemma}
\begin{proof}
For justification of steps, see comments below.
\begin{align}
&\sum_{i=k+1}^m \binom{2m-1}{2i-1}\biggl(\binom{2i}{2k}-m\binom{2i-1}{2k}\biggr)B_{2m-2i}^{(2m)}(m)\\\notag
&=\sum_{j=0}^{m-k-1} \binom{2m-1}{2m-2j-1}\biggl(\binom{2m-2j}{2k}-m\binom{2m-2j-1}{2k}\biggr)B_{2j}^{(2m)}(m)\\\notag
&=\sum_{j=0}^{2m-2k-1} \binom{2m-1}{2m-j-1}\biggl(\binom{2m-j}{2k}-m\binom{2m-j-1}{2k}\biggr)B_{j}^{(2m)}(m)\\\notag
&=\sum_{j=0}^{2m-2k-1} \biggl(\left(\frac{2m-j}{2m-j-2k}-m\right)\binom{2m-1}{2m-j-1}\binom{2m-j-1}{2k}\biggr)B_{j}^{(2m)}(m)\\\notag
&=\sum_{j=0}^{2m-2k-1} \biggl(\left(\frac{2m-j}{2m-j-2k}-m\right)\binom{2m-1}{2k}\binom{2m-2k-1}{j}\biggr)B_{j}^{(2m)}(m)\\\notag
&=\binom{2m-1}{2k}\sum_{j=0}^{2m-2k-1} \biggl(\left(\frac{2k}{2m-j-2k}+1-m\right)\binom{2m-2k-1}{j}\biggr)B_{j}^{(2m)}(m)\\\notag
&=\binom{2m-1}{2k}\biggl(\sum_{j=0}^{2m-2k-1} \frac{2k}{2m-j-2k}\binom{2m-2k-1}{j}B_{j}^{(2m)}(m)\\\notag
&+\sum_{j=0}^{2m-2k-1} (1-m)\binom{2m-2k-1}{j}B_{j}^{(2m)}(m)\biggr)\\\notag
&=\binom{2m-1}{2k}\biggl(\sum_{j=0}^{2m-2k-1} \frac{2k}{2m-2k}\binom{2m-2k}{j}B_{j}^{(2m)}(m)+\sum_{j=0}^{2m-2k-1} (1-m)\binom{2m-2k-1}{j}B_{j}^{(2m)}(m)\biggr)\\\notag
&=\binom{2m-1}{2k}\biggl(2k B_{2m-2k-1}^{(2m-1)}(m)+\sum_{j=0}^{2m-2k-1} (1-m)\binom{2m-2k-1}{j}B_{j}^{(2m)}(m)\biggr)\\\notag
&=\binom{2m-1}{2k}\biggl(2k B_{2m-2k-1}^{(2m-1)}(m)+\sum_{j=0}^{2m-2k-2} (1-m)\binom{2m-2k-1}{j}B_{j}^{(2m)}(m)+(1-m)B_{2m-2k-1}^{(2m)}(m)\biggr)\\\notag
&=\binom{2m-1}{2k}\biggl(2k B_{2m-2k-1}^{(2m-1)}(m)+(1-m)(2m-2k-1)B_{2m-2k-2}^{(2m-1)}(m)+(1-m)B_{2m-2k-1}^{(2m)}(m)\biggr)\\\notag
&=\binom{2m-1}{2(m-j-1)}\biggl(2(m-j-1) B_{2j+1}^{(2m-1)}(m)-(m-1)(2j+1)B_{2j}^{(2m-1)}(m)\biggr)\\\notag
\end{align}
Above, the first step is from setting $j=m-i$. The second is due to Proposition \ref{oddZero}. The third and fourth steps use the fact that
\begin{align}
\binom{n}{k}\binom{k}{j}&=\binom{n}{j}\binom{n-j}{k-j}\\\notag
\end{align}
Steps 8 and 10 use Prop. \ref{alphaMinus}. The last step uses Props. \ref{oddZero}, \ref{alphaMinus} and sets $k=m-j-1$.
Now, by Lemma \ref{reduced}, we are done.
\end{proof}
\begin{corollary}\label{stir}
For $m>k\ge 0$,
\begin{align}
\sum_{i=k+1}^m\sum_{r=2i-1}^{2m-1} \binom{r}{2i-1}\biggl(\binom{2i}{2k}-m\binom{2i-1}{2k}\biggr)s(2m,r+1)m^{r-2i+1}=0\\\notag
\end{align}

\end{corollary}
\begin{proof}
From Corollary \ref{Dprod}
\begin{align}
\sum_{i=k+1}^m& \binom{2m-1}{2i-1}\biggl(\binom{2i}{2k}-m\binom{2i-1}{2k}\biggr)B_{2m-2i}^{(2m)}(m)\\\notag
\end{align}
$$=\sum_{i=k+1}^m \binom{2m-1}{2i-1}\biggl(\binom{2i}{2k}-m\binom{2i-1}{2k}\biggr)\frac{(2m-2i)!}{(2m-1)!}\sum_{i=k+1}^m\frac{r!}{(r-2i+1)!}s(2m,r+1) m^{r-2i+1}$$
$$=\sum_{i=k+1}^m \sum_{r=2i-1}^{2m-1}\binom{2m-1}{2i-1}\biggl(\binom{2i}{2k}-m\binom{2i-1}{2k}\biggr)\frac{1}{\binom{2m-1}{2i-1}}\binom{r}{2i-1}s(2m,r+1) m^{r-2i+1}$$
$$=\sum_{i=k+1}^m \sum_{r=2i-1}^{2m-1}\binom{r}{2i-1}\biggl(\binom{2i}{2k}-m\binom{2i-1}{2k}\biggr)s(2m,r+1) m^{r-2i+1}$$
\end{proof}
\begin{corollary}
$(x-m)^2$ is a factor of
\begin{align}
p(x)&=\sum_{i=k+1}^m \binom{2m-1}{2i-1}\biggl(\binom{2i}{2k}-m\binom{2i-1}{2k}\biggr)B_{2m-2i}^{(2m)}(x)\\\notag
&=\sum_{i=k+1}^m\sum_{r=2i-1}^{2m-1} \binom{r}{2i-1}\biggl(\binom{2i}{2k}-m\binom{2i-1}{2k}\biggr)s(2m,r+1)x^{r-2i+1}\\\notag
\end{align}
\end{corollary}
\begin{proof}
We know $(x-m)$ is a factor of $p(x)$. From Corollary \ref{derivB}, 
\begin{align}
\frac{  \mathrm{  d  }   }{\mathrm{d}x  } p(x)&=\sum_{i=k+1}^m \binom{2m-1}{2i-1}\biggl(\binom{2i}{2k}-m\binom{2i-1}{2k}\biggr)(2m-2i)B_{2m-2i-1}^{(2m)}(x)\\\notag
\end{align}
But from Proposition \ref{oddZero}, $B_{2m-2i-1}^{(2m)}(m)$ are all zero, so $(x-m)$ is a factor of the derivative of $p(x)$.
\end{proof}
\begin{lemma}\label{prepprep}
\begin{align}
&\sum_{i=0}^{2m-1}\frac{1}{x^i}\binom{2m-1}{i}B_{2m-i-1}^{(2m)}(m)\left(\frac{1}{x}E_{i+1}(jx)-mE_i(jx)\right)\\\notag
&=\sum_{k=0}^{m-1}E_{2k+1}(0)\left(\binom{2m-1}{2k}B_{2m-2k-1}^{(2m)}(m+j)+(j-m)\binom{2m-1}{2k+1}B_{2m-2k-2}^{(2m)}(m+j)\right)x^{-2k-1}\\\notag
\end{align}
Note from Prop. \ref{oddZero}, the first sum equals
\begin{align}
&\sum_{i=0}^{m-1}\frac{1}{x^{2i+1}}\binom{2m-1}{2i+1}B_{2m-2i-2}^{(2m)}(m)\left(\frac{1}{x}E_{2i+2}(jx)-mE_{2i+1}(jx)\right)\\\notag
\end{align}
\end{lemma}
\begin{proof}
\begin{align}
&\sum_{i=0}^{2m-1}\frac{1}{x^i}\binom{2m-1}{i}B_{2m-i-1}^{(2m)}(m)\left(\frac{1}{x}E_{i+1}(jx)-mE_i(jx)\right)\\\notag
&=\sum_{i=0}^{2m-1}\frac{1}{x^i}\binom{2m-1}{i}B_{2m-i-1}^{(2m)}(m)\left(\frac{1}{x}\sum_{k=0}^{i+1}E_{k}(0)(jx)^{i+1-k}-m\sum_{k=0}^i\binom{i}{k}E_k(0)(jx)*{i-k}\right)\\\notag
&=\sum_{i=0}^{2m-1}\binom{2m-1}{i}B_{2m-i-1}^{(2m)}(m)\biggl(\biggl (\sum_{k=0}^{i}E_{k}(0)\left(j\binom{i+1}{k}-m\binom{i}{k}j^{i-k}\right)x^{-k}\biggr)+E_{i+1}(0)x^{-i-1}\biggr)\\\notag
&=\sum_{i=0}^{2m-1}\binom{2m-1}{i}B_{2m-i-1}^{(2m)}(m)\sum_{k=0}^{i}E_{k}(0)\left(    \left(j\frac{i+1}{i+1-k}-m\right)\binom{i}{k}\right)j^{i-k}x^{-k}
\end{align}
The end term went away since $E_{i+1}=0$ for $i$ odd, while $B_{2m-i-1}^{(2m)}(m)=0$ for $i$ even.
\begin{align}
&=\sum_{k=0}^{2m-1}\sum_{i=k}^{2m-1}\binom{2m-1}{i}\binom{i}{k}B_{2m-i-1}^{(2m)}(m)E_{k}(0)  \left(j\frac{i+1}{i+1-k}-m\right)j^{i-k}x^{-k}\\\notag
&=\sum_{k=0}^{2m-1}\sum_{i=k}^{2m-1}\binom{2m-1}{k}\binom{2m-1-k}{i-k}B_{2m-i-1}^{(2m)}(m)E_{k}(0) \left(j\frac{i+1}{i+1-k}-m\right)j^{i-k}x^{-k}\\\notag
&=\sum_{k=0}^{2m-1}\binom{2m-1}{k}E_{k}(0)\sum_{i=k}^{2m-1}\binom{2m-1-k}{i-k}B_{2m-i-1}^{(2m)}(m) \left(j\frac{i+1}{i+1-k}-m\right)j^{i-k}x^{-k}\\\notag
\end{align}
Next, separate out the $k=0$ term and change indexing to eliminate terms that equal zero.
\begin{align}\notag
&=\sum_{k=0}^{m-1}\binom{2m-1}{2k+1}E_{2k+1}(0)\sum_{i=2k+1}^{2m-1}\binom{2m-2-2k}{i-k}B_{2m-i-1}^{(2m)}(m)   \left(j\frac{i+1}{i+1-k}-m\right)j^{i-2k-1}x^{-2k-1}\\\notag
&+\sum_{i=0}^{2m-1}\binom{2m-1}{2m-1-i}B_{2m-i-1}^{(2m)}(m)( j-m ) j^i\\\notag
&=\sum_{k=0}^{m-1}\binom{2m-1}{2k+1}E_{2k+1}(0)\sum_{i=2k+1}^{2m-1}\binom{2m-2-2k}{i-k}B_{2m-i-1}^{(2m)}(m) \left(j\frac{i+1}{i+1-k}-m\right)j^{i-2k-1}x^{-2k-1}\\\notag
&+(j-m)\sum_{i=0}^{m-1}\binom{2m-1}{2m-2i-2}B_{2m-2i-2}^{(2m)}(m)j^{2i+1}\\\notag
&=\sum_{k=0}^{m-1}\binom{2m-1}{2k+1}E_{2k+1}(0)\sum_{i=2k+1}^{2m-1}\binom{2m-2-2k}{i-k}B_{2m-i-1}^{(2m)}(m) \left(j\frac{i+1}{i+1-k}-m\right)j^{i-2k-1}x^{-2k-1}\\\notag
&+(j-m)B_{2m-1}^{(2m)}(m+j)\\\notag
\end{align}
The last equivalence is by Prop. \ref{xplusy}., but this last term equals zero by Prop. \ref{oddZero}.
\begin{align}\notag
&=\sum_{k=0}^{m-1}\binom{2m-1}{2k+1}E_{2k+1}(0)\sum_{i=k}^{m-1}\binom{2m-2-2k}{2m-2i-2}B_{2m-2i-2}^{(2m)}(m)   \left(j\frac{2i+2}{2i+1-2k}-m\right)j^{2i-2k}x^{-2k-1}\\\notag
\end{align}
Replace $i$ with $m-i-1$.
\begin{align}\notag
&=\sum_{k=0}^{m-1}\binom{2m-1}{2k+1}E_{2k+1}(0)\sum_{i=0}^{m-k-1}\binom{2m-2-2k}{2i}B_{2i}^{(2m)}(m)   \left(j\frac{2(m-i)}{2(m-i)-1-2k}-m\right)j^{2m-2i-2k-2}x^{-2k-1}\\\notag
\end{align}
Again, parse out a factor and apply Prop. \ref{xplusy}.
\begin{align}\notag
&=\sum_{k=0}^{m-1}\binom{2m-1}{2k+1}E_{2k+1}(0)\sum_{i=0}^{m-k-1}\binom{2m-2-2k}{2i}B_{2i}^{(2m)}(m)  \left(j\frac{2(m-i)}{2(m-i)-1-2k}\right)j^{2m-2i-2k-2}x^{-2k-1}\\\notag
&-m\sum_{k=0}^{m-1}\binom{2m-1}{2k+1}E_{2k+1}(0)B_{2m-2k-2}^{(2m)}(m+j)x^{-2k-1}\\\notag
&=\sum_{k=0}^{m-1}\binom{2m-1}{2k+1}E_{2k+1}(0)\sum_{i=0}^{m-k-1}\binom{2m-2-2k}{2i}B_{2i}^{(2m)}(m)  j\left(1+\frac{1+2k}{2(m-i)-1-2k}\right)j^{2m-2i-2k-2}x^{-2k-1}\\\notag
&-m\sum_{k=0}^{m-1}\binom{2m-1}{2k+1}E_{2k+1}(0)B_{2m-2k-2}^{(2m)}(m+j)x^{-2k-1}\\\notag
&=\sum_{k=0}^{m-1}\binom{2m-1}{2k+1}E_{2k+1}(0)\sum_{i=0}^{m-k-1}\binom{2m-2-2k}{2i}B_{2i}^{(2m)}(m)  j\left(\frac{1+2k}{2(m-i)-1-2k}\right)j^{2m-2i-2k-2}x^{-2k-1}\\\notag
&+j\binom{2m-1}{2k+1}E_{2k+1}(0)B_{2m-2k-2}^{(2m)}(m+j)x^{-2k-1}-m\sum_{k=0}^{m-1}\binom{2m-1}{2k+1}E_{2k+1}(0)B_{2m-2k-2}^{(2m)}(m+j)x^{-2k-1}\\\notag
&=\sum_{k=0}^{m-1}\binom{2m-1}{2k+1}E_{2k+1}(0)\sum_{i=0}^{m-k-1}\left(\frac{1+2k}{2m-1-2k}\right)\binom{2m-1-2k}{2i}B_{2i}^{(2m)}(m)  j^{2m-2i-2k-2}x^{-2k-1}\\\notag
&+j\binom{2m-1}{2k+1}E_{2k+1}(0)B_{2m-2k-2}^{(2m)}(m+j)x^{-2k-1}-m\sum_{k=0}^{m-1}\binom{2m-1}{2k+1}E_{2k+1}(0)B_{2m-2k-2}^{(2m)}(m+j)x^{-2k-1}\\\notag
&=\sum_{k=0}^{m-1}\binom{2m-1}{2k+1}E_{2k+1}(0)\left(\frac{1+2k}{2m-1-2k}\right)B_{2m-2k-1}^{(2m)}(m+j) x^{-2k-1}\\\notag
&+j\binom{2m-1}{2k+1}E_{2k+1}(0)B_{2m-2k-2}^{(2m)}(m+j)x^{-2k-1}-m\sum_{k=0}^{m-1}\binom{2m-1}{2k+1}E_{2k+1}(0)B_{2m-2k-2}^{(2m)}(m+j)x^{-2k-1}\\\notag
&=\sum_{k=0}^{m-1}\binom{2m-1}{2k+1}E_{2k+1}(0)\biggl(\left(\frac{1+2k}{2m-1-2k}\right)B_{2m-2k-1}^{(2m)}(m+j)+(j-m) B_{2m-2k-2}^{(2m)}(m+j)\biggr)x^{-2k-1}\\\notag
&=\sum_{k=0}^{m-1}E_{2k+1}(0)\left(\binom{2m-1}{2k}B_{2m-2k-1}^{(2m)}(m+j)+(j-m)\binom{2m-1}{2k+1}B_{2m-2k-2}^{(2m)}(m+j)\right)x^{-2k-1}\\\notag
\end{align}
\end{proof}
\begin{lemma}\label{middletermprep}
\begin{align}
\sum_{i=0}^{2m-1}&\frac{1}{x^i}\binom{2m-1}{i}B_{2m-1-i}^{(2m)}(m)\biggl(\left(\frac{1}{x}E_{i+1}\left((j-1)x\right)-m E_i\left((j-1)x\right)\right)\\\notag
&-2\left(\frac{1}{x}E_{i+1}\left(j x\right)-m E_i\left(j x\right)\right)+\left(\frac{1}{x}E_{i+1}\left((j+1)x\right)-m E_i\left((j+1)x\right)\right)\biggr)\\\notag\\\notag
&=\sum_{k=0}^{m-2}\binom{2m-1}{2k+1}E_{2k+1}(0)(2m-2k-2)\biggl((1+2k)B_{2m-2k-3}^{(2m-2)}(m+j-1)\\\notag
&\  \  \  \  \  \  +(j-m)(2m-2k-3)B_{2m-2k-4}^{(2m-2)}(m+j-1)+B_{2m-2k-3}^{(2m-1)}(m+j)\\\notag
&\  \  \  \  \  \  +B_{2m-2k-3}^{(2m-1)}(m+j-1)\biggr) x^{-2k-1}\\\notag
\end{align}
\end{lemma}
\begin{proof}
\begin{align}
\sum_{i=0}^{2m-1}&\frac{1}{x^i}\binom{2m-1}{i}B_{2m-1-i}^{(2m)}(m)\biggl(\left(\frac{1}{x}E_{i+1}\left((j-1)x\right)-m E_i\left((j-1)x\right)\right)\\\notag
&-2\left(\frac{1}{x}E_{i+1}\left(j x\right)-m E_i\left(j x\right)\right)+\left(\frac{1}{x}E_{i+1}\left((j+1)x\right)-m E_i\left((j+1)x\right)\right)\biggr)\\\notag\\\notag
&=\sum_{k=0}^{m-2}\binom{2m-1}{2k+1}E_{2k+1}(0)(2m-2k-2)\biggl((1+2k)B_{2m-2k-3}^{(2m-2)}(m+j-1)\\\notag
&\  \  \  \  \  \  +(j-m)(2m-2k-1)B_{2m-2k-4}^{(2m-2)}(m+j-1)+B_{2m-2k-3}^{(2m-1)}(m+j)\\\notag
&\  \  \  \  \  \  +B_{2m-2k-3}^{(2m-1)}(m+j-1)\biggr) x^{-2k-1}\\\notag
\end{align}

By Lemma \ref{prepprep}, 
\begin{align}
\sum_{i=0}^{2m-1}&\frac{1}{x^i}\binom{2m-1}{i}B_{2m-1-i}^{(2m)}(m)\biggl(\left(\frac{1}{x}E_{i+1}\left((j-1)x\right)-m E_i\left((j-1)x\right)\right)\\\notag
&-2\left(\frac{1}{x}E_{i+1}\left(j x\right)-m E_i\left(j x\right)\right)+\left(\frac{1}{x}E_{i+1}\left((j+1)x\right)-m E_i\left((j+1)x\right)\right)\biggr)\\\notag
&=\sum_{k=0}^{m-1}\binom{2m-1}{2k+1}E_{2k+1}(0)\left(\frac{1+2k}{2m-2k-1}B_{2m-2k-1}^{(2m)}(m+j+1)+(j-m+1)B_{2m-2k-2}^{(2m)}(m+j+1)\right)\\\notag
&+\sum_{k=0}^{m-1}\binom{2m-1}{2k+1}E_{2k+1}(0)\left(\frac{1+2k}{2m-2k-1}B_{2m-2k-1}^{(2m)}(m+j-1)+(j-m-1)B_{2m-2k-2}^{(2m)}(m+j-1)\right)\\\notag
&+2\sum_{k=0}^{m-1}\binom{2m-1}{2k+1}E_{2k+1}(0)\left(\frac{1+2k}{2m-2k-1}B_{2m-2k-1}^{(2m)}(m+j)+(j-m)B_{2m-2k-2}^{(2m)}(m+j)\right)\\\notag
&=\sum_{k=0}^{m-1}\binom{2m-1}{2k+1}E_{2k+1}(0)\biggl(\frac{1+2k}{2m-2k-1}\biggl(B_{2m-2k-1}^{(2m)}(m+j+1)+B_{2m-2k-1}^{(2m)}(m+j-1)\\\notag
&-2B_{2m-2k-1}^{(2m)}(m+j)+((j-m+1)B_{2m-2k-2}^{(2m)}(m+j+1)+(j-m-1)B_{2m-2k-2}^{(2m)}(m+j-1)\\\notag
&-2(j-m)B_{2m-2k-2}^{(2m)}(m+j))\biggr)x^{-2k-1}\\\notag
\end{align}
Next we set up terms so that we can apply Prop. \ref{alphaMinus}.
\begin{align}\notag
&=\sum_{k=0}^{m-1}\binom{2m-1}{2k+1}E_{2k+1}(0)\biggl(\frac{1+2k}{2m-2k-1}\left((B_{2m-2k-1}^{(2m)}(m+j+1)-B_{2m-2k-1}^{(2m)}(m+j)\right)\\\notag
&-\left(B_{2m-2k-1}^{(2m)}(m+j)-B_{2m-2k-1}^{(2m)}(m+j-1)\right)\\\notag
&+(j-m)\left(B_{2m-2k-2}^{(2m)}(m+j+1)-B_{2m-2k-2}^{(2m)}(m+j)\right)\\\notag
&-(j-m)\left(B_{2m-2k-2}^{(2m)}(m+j)-B_{2m-2k-2}^{(2m)}(m+j-1)\right)\\\notag
&+B_{2m-2k-2}^{(2m)}(m+j+1)-B_{2m-2k-2}^{(2m)}(m+j)+B_{2m-2k-2}^{(2m)}(m+j)-B_{2m-2k-2}^{(2m)}(m+j-1)\biggr)x^{-2k-1}\\\notag
\end{align}
Note, the $k=m-1$ term is zero, since $B_0^{(\alpha)}(x))=1$ (apply Prop. \ref{alphaMinus} to first two differences). Apply Prop. \ref{alphaMinus}.
\begin{align}\notag
&=\sum_{k=0}^{m-1}\binom{2m-1}{2k+1}E_{2k+1}(0)\biggl((2m-2k-1)\biggl(\frac{1+2k}{2m-2k-1}(2m-2k-2)\biggl(B_{2m-2k-2}^{(2m-1)}(m+j)\\\notag
&-B_{2m-2k-2}^{(2m-1)}(m+j-1)\biggr)\biggr)+(j-m)(2m-2k-2)\left(B_{2m-2k-3}^{(2m-1)}(m+j)-B_{2m-2k-3}^{(2m-1)}(m+j-1)\right)\\\notag
&+(2m-2k-2)\left(B_{2m-2k-3}^{(2m-1)}(m+j)+B_{2m-2k-3}^{(2m-1)}(m+j-1)\right)\biggr)x^{-2k-1}\\\notag
&=\sum_{k=0}^{m-1}\binom{2m-1}{2k+1}E_{2k+1}(0)\biggl((2m-2k-1)\biggl(\frac{1+2k}{2m-2k-1}(2m-2k-2)\biggl(B_{2m-2k-2}^{(2m-1)}(m+j)\\\notag
&-B_{2m-2k-2}^{(2m-1)}(m+j-1)\biggr)\biggr)+(2m-2k-2)\biggl((j-m)(2m-2k-3)B_{2m-2k-4}^{(2m-2)}(m+j-1)\\\notag
&+B_{2m-2k-3}^{(2m-1)}(m+j)+B_{2m-2k-3}^{(2m-1)}(m+j-1)\biggr)\biggr)x^{-2k-1}\\\notag
&=\sum_{k=0}^{m-2}\binom{2m-1}{2k+1}E_{2k+1}(0)(2m-2k-2)\biggl((1+2k)B_{2m-2k-3}^{(2m-2)}(m+j-1)\\\notag
&\  \  \  \  \  \  +(j-m)(2m-2k-3)B_{2m-2k-4}^{(2m-2)}(m+j-1)+B_{2m-2k-3}^{(2m-1)}(m+j)\\\notag
&\  \  \  \  \  \  +B_{2m-2k-3}^{(2m-1)}(m+j-1)\biggr) x^{-2k-1}\\\notag
\end{align}
\end{proof}
In the following three lemmas the summation index $i$ includes both odd and even numbers in an interval. In fact, the even index $i$ terms are zero, from Proposition \ref{oddZero}, so we could have summed  using $2i-1$ in the formulas and reducing the upper limit of summation, e.g., $\sum_{i=0}^{2m-1}$ becomes $\sum_{i=1}^{m}$.
\begin{lemma}\label{firstterm}
\begin{align}
\sum_{i=0}^{2m-1}&x^{2m-i-1}\binom{2m-1}{i}B_{2m-1-i}^{(2m)}(m)\left(\frac{1}{x}E_{i+1}(x)-m E_i(x)\right)\\\notag
&=m\sum_{i=0}^{2m-1}x^{2m-i-1}\binom{2m-1}{i}B_{2m-1-i}^{(2m)}(m) E_i(0)\\\notag
\end{align}
\end{lemma}
\begin{proof}
We begin by re-indexing, as described above.
\begin{align}
\sum_{i=0}^{m-1}&\biggl(x^{2m-2i-2}\binom{2m-1}{2i+1}B_{2m-2i-2}^{(2m)}(m)\biggl(\frac{1}{x}E_{i2+2}(x)\\\notag
&-m E_{2i+1}(x)\biggr)-m x^{2m-2i-2}\binom{2m-1}{2i+1}B_{2m-2i-2}^{(2m)}(m)E_{2i+1}(0)\biggr)\\\notag
&=\sum_{i=0}^{m-1}\biggl(x^{2m-2i-2}\binom{2m-1}{2i+1}B_{2m-2i-2}^{(2m)}(m)\biggl(\frac{1}{x}\sum_{k=0}^{2i+2}\binom{2i+2}{k}E_k(0)x^{2i+2-k}\\\notag
&-m \sum_{k=0}^{2i+1}\binom{2i+1}{k}E_{k}(0)x^{2i+1-k}\biggr)-m x^{2m-2i-2}\binom{2m-1}{2i-1}B_{2m-2i-2}^{(2m)}(m)E_{2i+1}(0)\biggr)\\\notag
&=\sum_{i=0}^{m-1}\biggl(x^{2m-2i-2}\binom{2m-1}{2i+1}B_{2m-2i-2}^{(2m)}(m)\biggl(\frac{1}{x}\sum_{k=0}^{2i+1}\frac{2i+2}{2i+2-k}\binom{2i+1}{k}E_k(0)x^{2i+2-k}\\\notag
&-m \sum_{k=0}^{2i+1}\binom{2i+1}{k}E_{k}(0)x^{2i+1-k}\biggr)-m x^{2m-2i-2}\binom{2m-1}{2i-1}B_{2m-2i-2}^{(2m)}(m)E_{2i+1}(0)\biggr)\\\notag
&=\sum_{i=0}^{m-1}\biggl(x^{2m-2i-2}\binom{2m-1}{2i+1}B_{2m-2i-2}^{(2m)}(m)\biggl(\sum_{k=0}^{2i+1}\biggl(\left(\frac{2i+2}{2i+2-k}-m\right)\binom{2i+1}{k}E_k(0)x^{2i+1-k}\biggr)\\\notag
&-m E_{2i+1}(0)\biggr)\\\notag
&=\sum_{i=0}^{m-1}\biggl(x^{2m-2i-2}\binom{2m-1}{2i+1}B_{2m-2i-2}^{(2m)}(m)\biggl(\sum_{k=0}^{2i+1}\biggl(\left(\frac{2i+2}{2i+2-k}-m\right)\binom{2i+1}{k}E_k(0)x^{2i+1-k}\biggr)\\\notag
&-m E_{2i+1}(0)\biggr)\\\notag
&=\sum_{i=0}^{m-1}\biggl(x^{2m-2i-2}B_{2m-2i-2}^{(2m)}(m)\biggl(\sum_{k=0}^{2i+1}\biggl(\left(\frac{2i+2}{2i+2-k}-m\right)\binom{2m-1}{2i+1}\binom{2i+1}{k}E_k(0)x^{2i+1-k}\biggr)\\\notag
&-m\binom{2m-1}{2i+1} E_{2i+1}(0)\biggr)\\\notag
&=\sum_{i=0}^{m-1}\biggl(x^{2m-2i-2}B_{2m-2i-2}^{(2m)}(m)\biggl(\sum_{k=0}^{2i+1}\biggl(\left(\frac{2i+2}{2i+2-k}-m\right)\binom{2m-1}{k}\binom{2m-k-1}{2i+1-k}E_k(0)x^{2i+1-k}\biggr)\\\notag
&-m\binom{2m-1}{2i+1} E_{2i+1}(0)\biggr)\\\notag
&=\sum_{i=0}^{m-1}\biggl(x^{2m-2i-2}\biggl(\sum_{k=0}^{2i+1}\biggl(\left(\frac{2i+2}{2i+2-k}-m\right)\binom{2m-k-1}{2m-2i-2}B_{2m-2i-2}^{(2m)}(m)\binom{2m-1}{k}E_k(0)x^{2m-k-1}\biggr)\\\notag
&-m \binom{2m-1}{2i+1}B_{2m-2i-2}^{(2m)}(m)E_{2i+1}(0)x^{2m-2i-2}\biggr)\\\notag
&=\sum_{i=0}^{m-1}\biggl(x^{2m-2i-2}\biggl(\sum_{k=0}^{i}\biggl(\left(\frac{2i+2}{2i+1-2k}-m\right)\binom{2m-2k-2}{2m-2i-2}B_{2m-2i-2}^{(2m)}(m)\binom{2m-1}{2k+1}E_{2k+1}(0)x^{2m-2k-2}\biggr)\\\notag
&-m \binom{2m-1}{2i+1}B_{2m-2i-2}^{(2m)}(m)E_{2i+1}(0)x^{2m-2i-2}\biggr)\\\notag
\end{align}
Repalce $k$ by $m-1-k$ and then $i$ by $m-1-(k-i)$, this becomes
\begin{align}
&=\sum_{k=0}^{m-1}\sum_{i=0}^{k}\biggl(\left(\frac{2(m-k+i)}{2i+1}-m\right)\binom{2k}{2(k-i)}B_{2(k-i)}^{(2m)}(m)\binom{2m-1}{2(m-k)-1}E_{2(m-k)-1}(0)x^{2k}\biggr)\\\notag
&-m\sum_{k=0}^{m-1} \binom{2m-1}{2(m-k)-1}B_{2k}^{(2m)}(m)E_{2(m-k)-1}(0)x^{2k}\\\notag
\end{align}

This will be zero iff every coefficient of a power of $x$ is zero. Hence the problem is reduced to showing the following expression equals zero.
\begin{align}\notag
&\sum_{i=0}^{k}\biggl(\left(\frac{2(m-k+i)}{2i+1}-m\right)\binom{2k}{2(k-i)}B_{2(k-i)}^{(2m)}(m)\biggr)
-mB_{2k}^{(2m)}(m)\\\notag
\end{align}
Let's replace $i$ with $k-i$ and let $j=m-k$. This becomes
\begin{align}\notag
 \sum_{i=0}^{m-j}&\biggl(\left(\frac{2(m-i)}{2(m-j-i)+1}-m\right)\binom{2m-2j}{2i}B_{2i}^{(2m)}(m)\biggr)-m B_{2m-2j}^{(2m)}(m)\\\notag
 &=\sum_{i=0}^{m-j}\biggl(\left(\frac{2(m-i)}{2(m-j-i)+1}\right)\binom{2m-2j}{2i}B_{2i}^{(2m)}(m)\biggr)\\\notag
&-m \sum_{i=0}^{m-j}\biggl(\binom{2m-2j}{2i}B_{2i}^{(2m)}(m)\biggr)-m B_{2m-2j}^{(2m)}(m)\\\notag
 &=\sum_{i=0}^{m-j}\biggl(\left(\frac{2(m-i)}{2(m-j-i)+1}\right)\binom{2m-2j}{2i}B_{2i}^{(2m)}(m)\biggr)\\\notag
&-m \left(\sum_{i=0}^{m-j-1}\biggl(\binom{2m-2j}{2i}B_{2i}^{(2m)}(m)\biggr)+B_{2m-2j}^{(2m)}(m)\right)-m B_{2m-2j}^{(2m)}(m)\\\notag
 &=\sum_{i=0}^{m-j}\biggl(\left(\frac{2(m-i)}{2(m-j-i)+1}\right)\binom{2m-2j}{2i}B_{2i}^{(2m)}(m)\biggr)\\\notag
&-m \left((2m-2j)B_{2m-2j-1}^{(2m-1)}(m)+B_{2m-2j}^{(2m)}(m)\right)-m B_{2m-2j}^{(2m)}(m)\\\notag
&=\sum_{i=0}^{m-j}\biggl(\left(1+\frac{2j-1}{2(m-j-i)+1}\right)\binom{2m-2j}{2i}B_{2i}^{(2m)}(m)\biggr)\\\notag
&-m \left((2m-2j)B_{2m-2j-1}^{(2m-1)}(m)+B_{2m-2j}^{(2m)}(m)\right)-m B_{2m-2j}^{(2m)}(m)\\\notag
&=\sum_{i=0}^{m-j}\biggl(\binom{2m-2j}{2i}B_{2i}^{(2m)}(m)\biggr)+\sum_{i=0}^{m-j}\biggl(\left(\frac{2j-1}{2(m-j)+1}\right)\binom{2m-2j+1}{2i}B_{2i}^{(2m)}(m)\biggr)\\\notag
&-m \left((2m-2j)B_{2m-2j-1}^{(2m-1)}(m)+B_{2m-2j}^{(2m)}(m)\right)-m B_{2m-2j}^{(2m)}(m)\\\notag
&=\sum_{i=0}^{m-j}\biggl(\binom{2m-2j}{2i}B_{2i}^{(2m)}(m)\biggr)+(2j-1)B_{2m-2j}^{(2m-1)}(m)\\\notag
&-m \left((2m-2j)B_{2m-2j-1}^{(2m-1)}(m)+B_{2m-2j}^{(2m)}(m)\right)-m B_{2m-2j}^{(2m)}(m)\\\notag
&=\sum_{i=0}^{m-j-1}\binom{2m-2j}{2i}B_{2i}^{(2m)}(m)+B_{2m-2j}^{(2m)}(m)+(2j-1)B_{2m-2j}^{(2m-1)}(m)\\\notag
&-m \left((2m-2j)B_{2m-2j-1}^{(2m-1)}(m)+B_{2m-2j}^{(2m)}(m)\right)-m B_{2m-2j}^{(2m)}(m)\\\notag
&=(2m-2j)B_{2m-2j-1}^{(2m-1)}(m)+B_{2m-2j}^{(2m)}(m)+(2j-1)B_{2m-2j}^{(2m-1)}(m)\\\notag
&-m \left((2m-2j)B_{2m-2j-1}^{(2m-1)}(m)+B_{2m-2j}^{(2m)}(m)\right)-m B_{2m-2j}^{(2m)}(m)\\\notag
&=(2j-1)B_{2m-2j}^{(2m-1)}(m)-(m-1)(2m-2j)B_{2m-2j-1}^{(2m-1)}(m)-(2m-1)B_{2m-2j}^{(2m)}(m)\\\notag
\end{align}
This last expression is the subject of Lemma \ref{reduced}.
\end{proof}
\begin{lemma}\label{lastterm}
\begin{align}
\sum_{i=0}^{2m-1}&x^{2m-i}\binom{2m-1}{i}B_{2m-1-i}^{(2m)}(m)\left(\frac{1}{x}E_{i+1}\left(\frac{1}{2}\right)-m E_i\left(\frac{1}{2}\right)\right)\\\notag
&=\sum_{i=0}^{2m-1}x^{2m-i}\binom{2m-1}{i}B_{2m-1-i}^{(2m)}(m) \left(\frac{1}{x}E_{i+1}\left(x+\frac{1}{2}\right)-m E_i\left(x+\frac{1}{2}\right)\right)\\\notag
\end{align}
\end{lemma}
\begin{proof}
As mentioned above, this is equivalent to showing
\begin{align}
\sum_{i=1}^{m}&x^{2m-2i+1}\binom{2m-1}{2i-1}B_{2m-2i}^{(2m)}(m)\left(\frac{1}{x}E_{2i}\left(\frac{1}{2}\right)-m E_{2i-1}\left(\frac{1}{2}\right)\right)\\\notag
&=\sum_{i=1}^{m}x^{2m-2i+1}\binom{2m-1}{2i-1}B_{2m-2i}^{(2m)}(m) \left(\frac{1}{x}E_{2i}\left(x+\frac{1}{2}\right)-m E_{2i-1}\left(x+\frac{1}{2}\right)\right)\\\notag
\end{align}
\begin{align}
\sum_{i=1}^{m}&x^{2m-2i+1}\binom{2m-1}{2i-1}B_{2m-2i}^{(2m)}(m)\left(\frac{1}{x}E_{2i}\left(\frac{1}{2}\right)-m E_{2i-1}\left(\frac{1}{2}\right)\right)\\\notag
&=\sum_{i=1}^{m}\frac{1}{2^{2i}}x^{2m-2i}\binom{2m-1}{2i-1}B_{2m-2i}^{(2m)}(m)E(2i)\\\notag
\end{align}
while
\begin{align}
\sum_{i=1}^{m}&x^{2m-2i+1}\binom{2m-1}{2i-1}B_{2m-2i}^{(2m)}(m) \left(\frac{1}{x}E_{2i}\left(x+\frac{1}{2}\right)-m E_{2i-1}\left(x+\frac{1}{2}\right)\right)\\\notag
&=\sum_{i=1}^{m}x^{2m-2i+1}\binom{2m-1}{2i-1}B_{2m-2i}^{(2m)}(m) \biggl(\frac{1}{x}\sum_{k=0}^{2i}\binom{2i}{k}\frac{1}{2^k}E(k)x^{2i-k}\\\notag
&-m \sum_{k=0}^{2i-2}\binom{2i-1}{k}\frac{1}{2^k}E(k)x^{2i-k-1}\biggr)\\\notag
&=\sum_{i=1}^{m}x^{2m-2i+1}\binom{2m-1}{2i-1}B_{2m-2i}^{(2m)}(m) \biggl(\frac{1}{x}\biggl(\sum_{k=0}^{i}\binom{2i}{2k}\frac{1}{2^{2k}}E(2k)x^{2i-2k}\\\notag
&+\frac{1}{2^{2i}}E(2i)\biggr)-m \sum_{k=0}^{i-1}\binom{2i-1}{2k}\frac{1}{2^{2k}}E(2k)x^{2i-2k-1}\biggr)\\\notag
&=\sum_{i=1}^{m}\frac{1}{2^{2i}}x^{2m-2i}\binom{2m-1}{2i-1}B_{2m-2i}^{(2m)}(m)E(2i)\\\notag
&+\sum_{i=1}^{m}x^{2m-2i+1}\binom{2m-1}{2i-1}B_{2m-2i}^{(2m)}(m) \biggl(\frac{1}{x}\sum_{k=0}^{i-1}\binom{2i}{2k}\frac{1}{2^{2k}}E(2k)x^{2i-2k}\\\notag
&-m \sum_{k=0}^{i-1}\binom{2i-1}{2k}\frac{1}{2^{2k}}E(2k)x^{2i-2k-1}\biggr)\\\notag
\end{align}
So, the task is to show that the second sum (from $1$ to $m$) is zero. This in turn is equivalent to showing that the coefficients of powers of $x$ are zero. Looking at the coefficient of $x^{2m-2k-1}$, this amounts to showing
\begin{align}
\sum_{i=k+1}^m \binom{2m-1}{2i-1}\biggl(\binom{2i}{2k}-m\binom{2i-1}{2k}\biggr)B_{2m-2i}^{(2m)}(m)=0\\\notag
\end{align}
This is Lemma \ref{coeffx}.
\end{proof}

\begin{lemma}\label{middleterm}
\begin{align}
\sum_{i=0}^{2m-1}&\frac{1}{x^i}\binom{2m-1}{i}B_{2m-1-i}^{(2m)}(m)\biggl(\left(\frac{1}{x}E_{i+1}\left((j-1)x\right)-m E_i\left((j-1)x\right)\right)\\\notag
&-2\left(\frac{1}{x}E_{i+1}\left(j x\right)-m E_i\left(j x\right)\right)+\left(\frac{1}{x}E_{i+1}\left((j+1)x\right)-m E_i\left((j+1)x\right)\right)\biggr)\\\notag\\\notag
&=2m(2m-1)\sum_{i=0}^{2m-3}\frac{1}{x^i}\binom{2m-3}{i}B_{2m-3-i}^{(2m-2)}(m-1) \left(\frac{1}{x}E_{i+1}\left(jx\right)-m E_i\left(jx\right)\right)\\\notag
\end{align}
\end{lemma}
\begin{proof}
Using Lemmas \ref{prepprep} and \ref{middletermprep} we can re-write this equation as
\begin{align}
0&=\sum_{k=0}^{m-2}\binom{2m-1}{2k+1}E_{2k+1}(0)(2m-2k-2)\biggl((1+2k)B_{2m-2k-3}^{(2m-2)}(m+j-1)\\\notag
&\  \  \  \  \  \  +(j-m)(2m-2k-3)B_{2m-2k-4}^{(2m-2)}(m+j-1)+B_{2m-2k-3}^{(2m-1)}(m+j)\\\notag
&\  \  \  \  \  \  +B_{2m-2k-3}^{(2m-1)}(m+j-1)\biggr) x^{-2k-1}\\\notag
&-2m(2m-1)\sum_{k=0}^{m-2}E_{2k+1}(0)\biggl(\binom{2m-3}{2k}B_{2m-2k-3}^{(2(m-1))}(m+j-1)\\\notag
&+(j-m+1)\binom{2m-3}{2k+1}B_{2m-2k-4}^{(2(m-1)}(m+j-1)\biggr)x^{-2k-1}\\\notag
\end{align}
Since a polynomial in $x$ is zero iff each coefficient of a power of x$x$ is zero,. we can restrict to the coefficient of $x$ and remove non-zero factors.
\begin{align}
&(2m-2k-2)\biggl((1+2k)B_{2m-2k-3}^{(2m-2)}(m+j-1)\\\notag
&\  \  \  \  \  \  +(j-m)(2m-2k-3)B_{2m-2k-4}^{(2m-2)}(m+j-1)+B_{2m-2k-3}^{(2m-1)}(m+j)\\\notag
&\  \  \  \  \  \  +B_{2m-2k-3}^{(2m-1)}(m+j-1)\biggr) \\\notag
&-\frac{m}{m-1}\biggl((1+2k)B_{2m-2k-3}^{(2(m-1))}(m+j-1)\\\notag
&+(j-m+1)(2m-2k-3)B_{2m-2k-4}^{(2(m-1)}(m+j-1)\\\notag
\end{align}
We simplify to
\begin{align}
&(1+2k)B_{2m-2k-3}^{(2m-2)}(m+j-1) +j(2m-2k-3)B_{2m-2k-4}^{(2m-2)}(m+j-1)\\\notag
&+(m-1)B_{2m-2k-3}^{(2m-1)}(m+j) +B_{2m-2k-3}^{(2m-1)}(m+j-1) \\\notag
\end{align}
We can change variables to simplify more.
\begin{align}
&(2m-(2r+1)))B_{2r+1}^{(2m)}(m+j) +j(2r+1)B_{2r}^{(2m)}(m+j)\\\notag
&-mB_{2r+1}^{(2m+1)}(m+j) +(2r+1)B_{2m-2k}^{(2m)}(m+j) \\\notag
\end{align}

By Lemma \ref{reduced} we are done.
\end{proof}

\begin{lemma}\label{oddeven}
\begin{align}
&\sum_{j=1}^{m}\frac{1}{x^{2j-1}}(2^{2j}-1)\binom{2m}{2j}B_{2m-2j}^{(2m)}(m) B(2 j)\\\notag
&=-\frac{1}{x^{2m}}\sum_{k=0}^{m-1}x^{2k+1}\binom{2m-1}{2k}\Biggl(B_{2k}^{(2m)}(m)\Biggr)\Biggl(\frac{1}{x}E_{2m-2k}\left(x\right)-mE_{2m-2k-1}\left(x\right)\Biggr)\\\notag
\end{align}
\end{lemma}
\begin{proof}
By Lemma \ref{prepprep} we can change this into
\begin{align}
0&=\sum_{j=1}^{m}\frac{1}{x^{2j-1}}(2^{2j}-1)\binom{2m}{2j}B_{2m-2j}^{(2m)}(m) B(2 j)\\\notag
&+\sum_{k=0}^{m-1}E_{2m-2k-1}(0)\frac{1}{x^{2m-2k-1}}\biggl(\binom{2m-1}{2k+1}B_{2k+1}^{(2m)}(m+1)+(1-m)\binom{2m-1}{2k}B_{2k}^{(2m)}(m+1)\biggr)\\\notag
&=\sum_{j=1}^{m}\frac{1}{x^{2j-1}}(2^{2j}-1)\binom{2m}{2j}B_{2m-2j}^{(2m)}(m) B(2 j)\\\notag
&+\sum_{j=1}^{m}E_{2j-1}(0)\frac{1}{x^{2j-1}}\biggl(\binom{2m-1}{2m-2j+1}B_{2m-2j+1}^{(2m)}(m+1)+(1-m)\binom{2m-1}{2m-2j}B_{2m-2j}^{(2m)}(m+1)\biggr)\\\notag
&=\sum_{j=1}^{m}\frac{1}{x^{2j-1}}(2^{2j}-1)\binom{2m}{2j}B_{2m-2j}^{(2m)}(m) B(2 j)\\\notag
&+\sum_{j=1}^{m}E_{2j-1}(0)\frac{1}{x^{2j-1}}\biggl(\binom{2m-1}{2m-2j+1}B_{2m-2j+1}^{(2m)}(m+1)+(1-m)\binom{2m-1}{2m-2j}B_{2m-2j}^{(2m)}(m+1)\biggr)\\\notag
\end{align}
We can restrict to the coefficients of a power of $x$.
\begin{align}
0&=(2^{2j}-1)\binom{2m}{2j}B_{2m-2j}^{(2m)}(m) B(2 j)\\\notag
&+E_{2j-1}(0)\biggl(\binom{2m-1}{2m-2j+1}B_{2m-2j+1}^{(2m)}(m+1)+(1-m)\binom{2m-1}{2m-2j}B_{2m-2j}^{(2m)}(m+1)\biggr)\\\notag
&=(2^{2j}-1)\binom{2m}{2j}B_{2m-2j}^{(2m)}(m) B(2 j)\\\notag
&+\frac{1}{j}(1-2^{2j})B(2j)\biggl(\binom{2m-1}{2m-2j+1}B_{2m-2j+1}^{(2m)}(m+1)+(1-m)\binom{2m-1}{2m-2j}B_{2m-2j}^{(2m)}(m+1)\biggr)\\\notag
\end{align}
This reduces to showing
\begin{align}
mB_{2m-2j}^{(2m)}(m) -\frac{2j-1}{2m-2j+1}B_{2m-2j+1}^{(2m)}(m+1)-(1-m)B_{2m-2j}^{(2m)}(m+1)=0\\\notag
\end{align}
We use Prop. \ref{alphaMinus} to change this to
\begin{align}
mB_{2m-2j}^{(2m)}(m)& -\frac{2j-1}{2m-2j+1}\left((2m-2j+1)B_{2m-2j}^{(2m-1)}(m)+B_{2m-2j+1}^{(2m)}(m)\right)\\\notag
&-(1-m)\left((2m-2j)B_{2m-2j-1}^{(2m-1)}(m)+B_{2m-2j}^{(2m)}(m)\right)\\\notag
\end{align}
Note by Prop. \ref{oddZero} the $B_{2m-2j+1}^{(2m)}(m)$ term is zero.  After we collect terms and change variables, this becomes
\begin{align}
(2m-2k-1)B_{2k}^{(2m-1)}(m)+2k(1-m)B_{2k-1}^{(2m-1)}(m)+(1-2m)B_{2k}^{(2m)}(m)\\\notag
\end{align}
By Lemma \ref{reduced} we are done.
\end{proof}
\section{Negative Odd Powers}
\begin{proposition}\label{negodd}
Assume $r\in\mathbb{Z}$, $r$ odd. Regarding the $2^{n-2}x2^{n-2}$ matrices, $M_n$ which satisfy $$\frac{1}{(\sin(\frac{(2j-1)\pi}{2^n}))^r}=2^r \sum_{k=1}^{2^{n-2}}M_n(j,k)\sin(\frac{(2k-1)\pi}{2^n})$$ We summarize some facts. 
\begin{itemize}
\item the $M_n$ are normal matricies
\item the set of $M_n$ for a fixed $n$ and all $r$ commute
\item in a given $M_n$ each row contains the same elements with the order permuted and possible sign chage
\item for a fixed $n$ and all $r$ the permutations are the same for all $M_n$
\item for a fixed $n$ and all $r> 0$ the permutations and sign changes are the same for all $M_n$
\item for a fixed $n$ and all $r< 0$ the permutations and sign changes are the same for all $M_n$
\end{itemize}
\end{proposition}
\begin{proof}
Proofs can be found in ref. 2.
\end{proof}
\begin{proposition}\label{commute}
Let $$S_n=\sum_{k=1}^{2^{n-2}}\frac{1}{\sin^r(\frac{(2k-1)\pi}{2^n})}$$
Regarding the $2^{n-2}x2^{n-2}$ matrices, $M_n$ which satisfy $$\frac{1}{(\sin(\frac{(2j-1)\pi}{2^n}))^r}=2^r \sum_{k=1}^{2^{n-2}}M_n(j,k)\sin(\frac{(2k-1)\pi}{2^n})\label{mequa}$$
then $$S_n=2^{r-1} \sum_{j=1}^{2^{n-2}}M_n(1,j)\frac{1}{\sin(\frac{(2j-1)\pi}{2^n})}$$
\end{proposition}
\begin{proof}
We will show how the proof goes with an example having $n=4$, $r=3$.
\begin{equation}\notag
S_4=\left(
\begin{array}{cccc}
 1 & 1 & 1 & 1 \\
\end{array}
\right)
\left(
\begin{array}{c}
\sin^{-3}\left(\frac{\pi}{16}\right)   \\
\sin^{-3}\left(\frac{3\pi}{16}\right)  \\
\sin^{-3}\left(\frac{5\pi}{16}\right)  \\
\sin^{-3}\left(\frac{7\pi}{16}\right)  \\
\end{array}
\right)
\end{equation}
\begin{equation}\notag
\left(
\begin{array}{c}
\sin^{-3}\left(\frac{\pi}{16}\right)   \\
\sin^{-3}\left(\frac{3\pi}{16}\right)  \\
\sin^{-3}\left(\frac{5\pi}{16}\right)  \\
\sin^{-3}\left(\frac{7\pi}{16}\right)  \\
\end{array}
\right)
=8\left(
\begin{array}{cccc}
 2 & 5 & 7 & 8 \\
 7 & 2 & -8 & 5 \\
 5 & 8 & 2 & -7 \\
 -8 & 7 & -5 & 2 \\
\end{array}
\right)
\left(
\begin{array}{c}
\sin\left(\frac{\pi}{16}\right)   \\
\sin\left(\frac{3\pi}{16}\right)  \\
\sin\left(\frac{5\pi}{16}\right)  \\
\sin\left(\frac{7\pi}{16}\right)  \\
\end{array}
\right)
\end{equation}
Multiplying both sides of the above equation on the left by the matrix $M$ for $n=4$, $r=1$
\begin{equation}\notag
\left(
\begin{array}{cccc}
 1 & 1 & 1 & 1 \\
 1 & 1 & -1 & 1 \\
 1 & 1 & 1 & -1 \\
-1 & 1 & -1 & 1 \\
\end{array}
\right)
\end{equation}
We see that $S_n$ occupies the first entry in the resulting vector. However, we know from Proposition \ref {commute} that these matricies commute. 
\begin{equation}\notag
\left(
\begin{array}{cccc}
 1 & 1 & 1 & 1 \\
 1 & 1 & -1 & 1 \\
 1 & 1 & 1 & -1 \\
-1 & 1 & -1 & 1 \\
\end{array}
\right)
\left(
\begin{array}{c}
\sin^{-3}\left(\frac{\pi}{16}\right)   \\
\sin^{-3}\left(\frac{3\pi}{16}\right)  \\
\sin^{-3}\left(\frac{5\pi}{16}\right)  \\
\sin^{-3}\left(\frac{7\pi}{16}\right)  \\
\end{array}
\right)
=8\left(
\begin{array}{cccc}
 2 & 5 & 7 & 8 \\
 7 & 2 & -8 & 5 \\
 5 & 8 & 2 & -7 \\
 -8 & 7 & -5 & 2 \\
\end{array}
\right)
\left(
\begin{array}{cccc}
 1 & 1 & 1 & 1 \\
 1 & 1 & -1 & 1 \\
 1 & 1 & 1 & -1 \\
-1 & 1 & -1 & 1 \\
\end{array}
\right)
\left(
\begin{array}{c}
\sin\left(\frac{\pi}{16}\right)   \\
\sin\left(\frac{3\pi}{16}\right)  \\
\sin\left(\frac{5\pi}{16}\right)  \\
\sin\left(\frac{7\pi}{16}\right)  \\
\end{array}
\right)
\end{equation}
\begin{equation}\notag
=4\left(
\begin{array}{cccc}
 2 & 5 & 7 & 8 \\
 7 & 2 & -8 & 5 \\
 5 & 8 & 2 & -7 \\
 -8 & 7 & -5 & 2 \\
\end{array}
\right)
\left(
\begin{array}{c}
\sin^{-1}\left(\frac{\pi}{16}\right)   \\
\sin^{-1}\left(\frac{3\pi}{16}\right)  \\
\sin^{-1}\left(\frac{5\pi}{16}\right)  \\
\sin^{-1}\left(\frac{7\pi}{16}\right)  \\
\end{array}
\right)
\end{equation}\notag
The general case uses the same argument.
\end{proof}
The objective of this section is to state and prove the formulas that comprise the matricies $M_n$. As stated above, the matricies are completely determined by their first row. In the above proposition, for the case of odd power $r=2m+1$, the equation of the first row looks like
$$\frac{1}{(\sin(\frac{\pi}{2^n}))^{2m+1}}=2^{2m+1} \sum_{k=1}^{2^{n-2}}M_n(1,k)\sin(\frac{(2k-1)\pi}{2^n})$$
The strategy is to use induction on $m$. Multiplying both sides of this equation by $\frac{1}{(\sin(\frac{\pi}{2^n}))^{2}}$ corresponds to increasing $m$ by $1$. Multiplying both sides by  $(\sin(\frac{\pi}{2^n}))^{2}$ corresponds to decreasing $m$. Hence we begin with lemmas that describe the effect of those multiplications.
\begin{lemma}\label{sinequotient}
\begin{align}
\frac{\sin(\frac{(2j-1)\pi}{2^n})}{\left(\sin(\frac{\pi}{2^n})\right)^2}&=\sum_{i=1}^{2^{n-1}}i\sin(\frac{(2j-(2^n-2i+1))\pi}{2^n})+\sum_{i=1}^{2^{n-1}-1}i\sin(\frac{(2j+(2^n-2i-1))\pi}{2^n})\\\notag
\end{align}
\end {lemma}
\begin{proof}
The proof uses the identity $$\sin(\theta)\cos(\phi)=\frac{1}{2}\left(\sin(\theta+\phi)+\sin(\theta-\phi)\right)$$
In the denominator we use
$$\sin(\theta)\cos(\theta)=\frac{1}{2}\left(\sin(2\theta)\right)$$
to successively multiply the value of the argument of $\sin$ by 2 until we reach $\frac{\pi}{4}$, thereby clearing the denominator. The same value multiplies the numerator where we use the trig identity to combine products of sine and cosine.
\begin{align}
\frac{\sin(\frac{(2j-1)\pi}{2^n})}{\left(\sin(\frac{\pi}{2^n})\right)^2}&=\frac{\sin(\frac{(2j-1)\pi}{2^n})\prod_{k=0}^{n-2}\left(\cos(\frac{\pi}{2^{n-k}})\right)^2 }{\left(\sin(\frac{\pi}{2^n})\right)^2\prod_{k=0}^{n-2}\left(\cos(\frac{\pi}{2^{n-k}})\right)^2}\\\notag
\end{align}

By induction on $n$. The first few values of $n$ are easily checked. Assume true for $n$ and consider $n+1$.
\begin{align}
\frac{\sin(\frac{(2j-1)\pi}{2^{n+1}})}{\left(\sin(\frac{\pi}{2^{n+1}})\right)^2}&=\frac{\sin(\frac{(2j-1)\pi}{2^{n+1)}})\left(\cos(\frac{\pi}{2^{n+1}})\right)^2 }{ \left(\sin(\frac{\pi}{2^{n+1}})\right)^2\left(\cos(\frac{\pi}{2^{n+1}})\right)^2  }\\\notag
&=\frac{4\sin(\frac{(2j-1)\pi}{2^{n+1)}})\left(\cos(\frac{\pi}{2^{n+1}})\right)^2 }{ \left(\sin(\frac{\pi}{2^{n}})\right)^2  }\\\notag
&=\frac{2\left(\sin(\frac{j\pi}{2^{n}})+\sin(\frac{(j-1)\pi}{2^{n}})\right)\cos(\frac{\pi}{2^{n+1}}) }{ \left(\sin(\frac{\pi}{2^{n}})\right)^2  }\\\notag
&=\frac{2\left(\sin(\frac{(2(j+1)/2-1)\pi}{2^{n}})+\sin(\frac{(2(j/2)-1)\pi}{2^{n}})\right)\cos(\frac{\pi}{2^{n+1}}) }{ \left(\sin(\frac{\pi}{2^{n}})\right)^2  }\\\notag
\end{align}
\begin{align}
&=\left(\sum_{i=1}^{2^{n-1}}i\sin(\frac{(j+1-(2^n-2i+1))\pi}{2^n})+\sum_{i=1}^{2^{n-1}-1}i\sin(\frac{(j+1+(2^n-2i-1))\pi}{2^n})\right)\cos(\frac{\pi}{2^{n+1}})\\\notag
&+\left(\sum_{i=1}^{2^{n-1}}i\sin(\frac{(j-(2^n-2i+1))\pi}{2^n})+\sum_{i=1}^{2^{n-1}-1}i\sin(\frac{(j+(2^n-2i-1))\pi}{2^n})\right)\cos(\frac{\pi}{2^{n+1}})\\\notag
&=\sum_{i=1}^{2^{n-1}}i\biggl(\sin(\frac{(2(j-2^n+2i)+1)\pi}{2^{n+1}})+\sin(\frac{(2(j-2^n+2i)-1)\pi}{2^{n+1}})\biggr)\\\notag
&+\sum_{i=1}^{2^{n-1}-1}i\biggl(\sin(\frac{(2(j+2^n-2i)+1)\pi}{2^{n+1}})+\sin(\frac{(2(j+2^n-2i)-1)\pi}{2^{n+1}})\biggr)\\\notag
&+\sum_{i=1}^{2^{n-1}}i\biggl(\sin(\frac{(2(j-2^n+2i-1)+1)\pi}{2^{n+1}})+\sin(\frac{(2(j-2^n+2i-1)-1)\pi}{2^{n+1}})\biggr)\\\notag
&+\sum_{i=1}^{2^{n-1}-1}i\biggl(\sin(\frac{(2(j+2^n-2i-1)+1)\pi}{2^{n+1}})+\sin(\frac{(2(j+2^n-2i-1)-1)\pi}{2^{n+1}})\biggr)\\\notag
&=\sum_{i=1}^{2^{n-1}}i\biggl(\sin(\frac{(2(j-2^n+2i)+1)\pi}{2^{n+1}})+2\sin(\frac{(2(j-2^n+2i)-1)\pi}{2^{n+1}})\biggr)\\\notag
&+\sum_{i=1}^{2^{n-1}-1}i\biggl(\sin(\frac{(2(j+2^n-2i)+1)\pi}{2^{n+1}})+2\sin(\frac{(2(j+2^n-2i)-1)\pi}{2^{n+1}})\biggr)\\\notag
&+\sum_{i=1}^{2^{n-1}}i\biggl(\sin(\frac{(2(j-2^n+2i)-3)\pi}{2^{n+1}})\biggr)+\sum_{i=1}^{2^{n-1}-1}i\biggl(\sin(\frac{(2(j+2^n-2i)-3)\pi}{2^{n+1}})\biggr)\\\notag
&=\sum_{i=1}^{2^{n-1}}i\biggl(\sin(\frac{(2j-2^{n+1}+(4i+1))\pi}{2^{n+1}})+2\sin(\frac{(2j-2^{n+1}+(4i-1))\pi}{2^{n+1}})\biggr)\\\notag
&+\sum_{i=1}^{2^{n-1}-1}i\biggl(\sin(\frac{(2j+2^{n+1}-(4i-1))\pi}{2^{n+1}})+2\sin(\frac{(2j+2^{n+1}-(4i+1))\pi}{2^{n+1}})\biggr)\\\notag
&+\sum_{i=1}^{2^{n-1}}i\biggl(\sin(\frac{(2j-2^{n+1}+(4i-3))\pi}{2^{n+1}})\biggr)+\sum_{i=1}^{2^{n-1}-1}i\biggl(\sin(\frac{(2j+2^{n+1}-(4i+3))\pi}{2^{n+1}})\biggr)\\\notag
&=\sum_{i=1}^{2^{n-1}}2i\biggl(\sin(\frac{(2j-2^{n+1}+(4i-1))\pi}{2^{n+1}})\biggr)+\sum_{i=1}^{2^{n-1}-1}2i\biggl(\sin(\frac{(2j+2^{n+1}-(4i+1))\pi}{2^{n+1}})\biggr)\\\notag
&+\sum_{i=1}^{2^{n-1}}i\biggl(\sin(\frac{(2j-2^{n+1}+(4i-3))\pi}{2^{n+1}})\biggr)+\sum_{i=1}^{2^{n-1}-1}i\sin(\frac{(2j+2^{n+1}-(4i+3))\pi}{2^{n+1}})\biggr)\\\notag
&+\sum_{i=1}^{2^{n-1}}i\biggl(\sin(\frac{(2j-2^{n+1}+(4i+1))\pi}{2^{n+1}})\biggr)+\sum_{i=1}^{2^{n-1}-1}i\biggl(\sin(\frac{(2j+2^{n+1}-(4i-1))\pi}{2^{n+1}})\biggr)\\\notag
\end{align}
Note if you repace $i$ with $i-1$ in $4i+1$, you get $4i-3$. 
Hence the sum equals
\begin{align}
&\sum_{i=1}^{2^{n-1}}2i\biggl(\sin(\frac{(2j-2^{n+1}+(4i-1))\pi}{2^{n+1}})\biggr)+\sum_{i=1}^{2^{n-1}-1}2i\biggl(\sin(\frac{(2j+2^{n+1}-(4i+1))\pi}{2^{n+1}})\biggr)\\\notag
&+\sum_{i=1}^{2^{n-1}}i\biggl(\sin(\frac{(2j-2^{n+1}+(4i-3))\pi}{2^{n+1}})\biggr)+\sum_{i=2}^{2^{n-1}}(i-1)\biggl(\sin(\frac{(2j+2^{n+1}-(4i-1))\pi}{2^{n+1}})\biggr)\\\notag
&+\sum_{i=2}^{2^{n-1}+1}(i-1)\biggl(\sin(\frac{(2j-2^{n+1}+(4i-3))\pi}{2^{n+1}})\biggr)+\sum_{i=1}^{2^{n-1}-1}i\biggl(\sin(\frac{(2j+2^{n+1}-(4i-1))\pi}{2^{n+1}})\biggr)\\\notag
&=\sum_{i=1}^{2^{n}}i\sin(   \frac{(2j-(2^{n+1}-2i+1))\pi}{2^{n+1}}  )+\sum_{i=1}^{2^{n}-1}i\sin(   \frac{(2j+(2^{n+1}-2i-1))\pi}{2^{n+1}}  )\\\notag
\end{align}
\end{proof}
\begin{corollary}
\begin{align}
\frac{1}{\left(\sin(\frac{\pi}{2^n})\right)^2}\sum_{j=1}^{2^{n-2}}&a(j)\sin(\frac{(2j-1))\pi}{2^n})=\\\notag
&\sum_{j=1}^{2^{n-2}}\biggl(\sum_{i=1}^j2(2i-1)a(i)+2(2j-1)\sum_{i=j+1}^{2^{n-2}}a(i)\biggr)\sin(\frac{(2j-1)\pi}{2^n})\\\notag
\end{align}
\end {corollary}
\begin{proof}
This can be proven by induction using the same approach as in the previous lemma.
\end{proof}
\begin{corollary}
Assume the recursion
\begin{align}
a(j,m)=\frac{1}{2}(\sum_{i=1}^j(2i-1)a(i,m-1)+(2j-1)\sum_{i=j+1}^{2^{n-2}}a(i,m-1))\\\notag
\end{align}
with $a(j,0)=1$ for all $j$.
And assume the $2^{n-2}x2^{n-2}$ matrix $M_n$ satisfies
\begin{align}
 \frac{1}{(\sin(\frac{(2j-1)\pi}{2^n}))^{2m+1}}&=2^{2m+1} \sum_{k=1}^{2^{n-2}}M_n(j,k)\sin(\frac{(2k-1)\pi}{2^n})\\\notag
\end{align}
Then $$M_n(1,j)=a(j,m)$$
\end {corollary}
\begin{proof}
This follows from the previous lemma since the matrix corresponding to $m=0$ has $M(1,j)=1$ for all $j$.
\end{proof}
\begin{lemma}
\begin{align}
\sin\left(\frac{(2j-1)\pi}{2^n}\right)\left(\sin\left(\frac{\pi}{2^n}\right)\right)^2&=\frac{1}{4}\left(-\sin\left(\frac{(2j-3)\pi}{2^n}\right)+2\sin\left(\frac{(2j-1)\pi}{2^n}\right)-\sin\left(\frac{(2j+1)\pi}{2^n}\right)\right)\\\notag
\end{align}
\end{lemma}
\begin{proof}
\begin{align}
\sin\left(\frac{(2j-1)\pi}{2^n}\right)\left(\sin\left(\frac{\pi}{2^n}\right)\right)^2&=\sin\left(\frac{(2j-1)\pi}{2^n}\right)\left(1-\left(\cos\left(\frac{\pi}{2^n}\right)^2\right)\right)\\\notag
\end{align}
Then use $$\sin(\theta)\cos(\phi)=\frac{1}{2}\left(\sin(\theta+\phi)+\sin(\theta-\phi)\right)$$
\end{proof}
\begin{corollary}
\begin{align}
\left(\sin(\frac{\pi}{2^n})\right)^2&\sum_{j=1}^{2^{n-2}}a(j)\sin(\frac{(2j-1)\pi}{2^n})=\frac{1}{4}\biggl(\left(3a(1)-a(2)\right)\sin\left(\frac{\pi}{2^n})\right)+\\\notag
&\sum_{j=3}^{2^{n-2}-1}\left(-a(j-1)+2a(j)-a(j+1)\right)\sin\left(\frac{(2j-1)\pi}{2^n}\right)\\\notag
&+\left(-a(2^{n-2}-1)+a(2^{n-2})\right)\sin\left(\frac{(2^{n-1}-1)\pi}{2^n}\right)\biggr)\\\notag
\end{align}
\end {corollary}
\begin{proof}
This can be proven by induction using the same approach as in the previous lemma.
\end{proof}
\begin{corollary}\label{induct}

If $a(0)$ and $a(2^{n-2}+1)$ are defined and $a(0)=-a(1)$ and $a(2^{n-2}+1)=a(2^{n-2})$, then
\begin{align}
\left(\sin(\frac{\pi}{2^n})\right)^2&\sum_{j=1}^{2^{n-2}}a(j)\sin\left(\frac{(2j-1)\pi}{2^n}\right)=\\\notag
&\frac{1}{4}\sum_{j=1}^{2^{n-2}}\left(-a(j-1)+2a(j)-a(j+1)\right)\sin\left(\frac{(2j-1)\pi}{2^n}\right)\\\notag
\end{align}
\end {corollary}

\begin{proposition}\label{Mmatrix}

Assume $m\in\mathbb{N}$. Regarding the $2^{n-2}x2^{n-2}$ matrices, $M_n$ which satisfy $$\frac{1}{(\sin(\frac{(2j-1)\pi}{2^n}))^{2m+1}}=2^{2m+1} \sum_{k=1}^{2^{n-2}}M_n(j,k)\sin(\frac{(2k-1)\pi}{2^n})$$
Let $E_j(x)$ be the j-th Euler polynomial and let $B_n^{(\alpha)}(x)$ be the generalized Bernoulli polynomial of degree $n$ in $x$ then
\begin{align}
M_n(1,j)&=(-1)^m\frac{2^{2m(n-1)}}{(2m)!}\sum_{k=0}^{m-1}\Biggl(2^{(n-1)(-2k-1)}\binom{2m-1}{2k}\\\notag
&\Biggl(B_{2k}^{(2m)}(m)\Biggr)\Biggl(2^{n-1}E_{2m-2k}\left(\frac{j}{2^{n-1}}\right)-mE_{2m-2k-1}\left(\frac{j}{2^{n-1}}\right)\Biggr)\Biggr)\\\notag
\end{align}
\end{proposition}
\begin{proof}
The proof is by induction on $m$. Note that if we demonstrate the induction on $m$, we are done since we have previously proven the case $m=1$ is true for arbitrary $n>2$. The case of $m=1$ is presented in ref. 1. Assume true for $m$. Keep in mind, the matrix $M_n$ is completely determined by the first row. By the preceding corollary
\begin{align}
\frac{1}{(\sin(\frac{\pi}{2^n}))^{2(m+1)+1}}&=\frac{1}{\left(\sin(\frac{\pi}{2^n})\right)^2}\frac{1}{(\sin(\frac{\pi}{2^n}))^{2m+1}}\\\notag
&=2^{2m+1}\sum_{k=1}^{2^{n-2}}\biggl(\sum_{i=1}^k2(2i-1)M_n(1,i)+2(2k-1)\sum_{i=k+1}^{2^{n-2}}M_n(1,i)\biggr)\sin(\frac{(2k-1)\pi}{2^n})\\\notag
\end{align}
This tells us that the matrix $\bar{M_n}$ corresponding to the power $2m+3$ has first row $$\bar{M}_n(1,k)=\frac{1}{2}\bigl(\sum_{i=1}^k(2i-1)M_n(1,i)+(2k-1)\sum_{i=k+1}^{2^{n-2}}M_n(1,i))\bigr)$$
Alternatively, if the matrix $\bar{M_n}$ corresponds to the power $2m-1$, our  Corollary \ref{induct} suggests we show
\begin{align}
M_n(1,0)&=-M_n(1,1)\\\notag
M_n(1,2^{n-2})&=M_n(1,2^{n-2}+1)\\\notag
\bar{M}_{n}(1,j)&=\frac{1}{4}\biggl(-M_n(1,j-1)+2M_n(1,j)-M_n(1,j+1)\biggr)\\\notag
\end{align}
The first step is essentially Lemma \ref{firstterm}. The second is Lemma \ref{lastterm}. The third is Lemma \ref{middleterm}, where we substitute $\frac{1}{2^{n-1}}$ for $x$.

\end{proof}

\section{Sums Of Negative Odd Powers}

Lets write our example equations in a more revealing form.
\begin{align}\notag
S(3,n)&=\sum_{j=1}^{2^{n-2}}\biggl(2^{2n-2} 2\left(-\left(\frac{j}{2^{n-1}}\right)^2+\frac{j}{2^{n-1}} \right) +2^{n-1} 2\left(\frac{j}{2^{n-1}}-\frac{1}{2}\right)\biggr)\frac{1}{\sin[\frac{(2j-1)\pi}{2^n}]}\\\notag
S(5,n)&=\frac{1}{3}\sum_{j=1}^{2^{n-2}} \biggl(2^{4n-4}2\left(  \left(\frac{j}{2^{n-1}}\right)^4   -2\left(\frac{j}{2^{n-1}}\right)^3   +\frac{j}{2^{n-1}}   \right)\\\notag
&+2^{3n-3}(-4)\left(  \left(\frac{j}{2^{n-1}}\right)^3  -\frac{3}{2}\left(\frac{j}{2^{n-1}}\right)^2  +\frac{1}{4}  \right)\\\notag
&+2^{2n-2} 2\left(  -\left(\frac{j}{2^{n-1}}\right)^2  +\frac{j}{2^{n-1}}\right)+2^{n-1} 4\left(\frac{j}{2^{n-1}}-\frac{1}{2}\right)  \biggr)\frac{1}{\sin[\frac{(2j-1)\pi}{2^n}]}\\\notag
S(7,n)&=\frac{1}{45}\sum_{j=1}^{2^{n-2}}\biggl(2^{6n-6}  (-4)\left(   \left( \frac{j}{2^{n-1}}\right)^6-3\left(\frac{j}{2^{n-1}}\right)^5+5\left(\frac{j}{2^{n-1}}\right)^3-3\frac{j}{2^{n-1}}  \right)\\\notag
&+2^{5n-5}12\left(   \left( \frac{j}{2^{n-1}}\right)^5-\frac{5}{2}\left( \frac{j}{2^{n-1}}\right)^4+\frac{5}{2}\left( \frac{j}{2^{n-1}}\right)^2-\frac{1}{2}\right) \\\notag
&+2^{4n-4}20\left(  \left(\frac{j}{2^{n-1}}\right)^4   -2\left(\frac{j}{2^{n-1}}\right)^3   +\frac{j}{2^{n-1}}   \right)+2^{3n-3}(-60)\left(  \left(\frac{j}{2^{n-1}}\right)^3  -\frac{3}{2}\left(\frac{j}{2^{n-1}}\right)^2  +\frac{1}{4}  \right)\\\notag
&+2^{2n-2} 16\left(  -\left(\frac{j}{2^{n-1}}\right)^2  +\frac{j}{2^{n-1}}\right)+2^{n-1} 48\left(\frac{j}{2^{n-1}}-\frac{1}{2}\right)  \biggr)\frac{1}{\sin[\frac{(2j-1)\pi}{2^n}]}\\\notag
\end{align}

These are the Euler polynomials. Let $E_n(x)$ be the $n$-th Euler polynomial. Then these equations can be written

\begin{align}\notag
S(3,n)&=2\sum_{j=1}^{2^{n-2}}\biggl(2^{2n-2}(-1) E_2\left(\frac{j}{2^{n-1}}\right) +2^{n-1} E_1\left(\frac{j}{2^{n-1}}\right)\biggr)\frac{1}{\sin[\frac{(2j-1)\pi}{2^n}]}\\\notag
S(5,n)&=\frac{2}{3}\sum_{j=1}^{2^{n-2}} \biggl(2^{4n-4}E_4\left(\frac{j}{2^{n-1}}\right)+2^{3n-3}(-2)E_3\left(\frac{j}{2^{n-1}}\right)\\\notag
&+2^{2n-2} (-1)E_2\left(\frac{j}{2^{n-1}}\right)+2^{n-1} 2E_1\left(\frac{j}{2^{n-1}}\right) \biggr)\frac{1}{\sin[\frac{(2j-1)\pi}{2^n}]}\\\notag
S(7,n)&=\frac{4}{45}\sum_{j=1}^{2^{n-2}}\biggl(2^{6n-6}  (-1)E_6\left(\frac{j}{2^{n-1}}\right)+2^{5n-5}3E_5\left(\frac{j}{2^{n-1}}\right)\\\notag
&+2^{4n-4}5E_4\left(\frac{j}{2^{n-1}}\right)+2^{3n-3}(-15)E_3\left(\frac{j}{2^{n-1}}\right)\\\notag
&+2^{2n-2}(- 4)E_2\left(\frac{j}{2^{n-1}}\right)+2^{n-1} 12E_1\left(\frac{j}{2^{n-1}}\right)  \biggr)\frac{1}{\sin[\frac{(2j-1)\pi}{2^n}]}\\\notag
S(9,n)&=\frac{2}{315}\sum_{j=1}^{2^{n-2}}\biggl(2^{8n-8}  E_8\left(\frac{j}{2^{n-1}}\right)+2^{7n-7}(-4)E_7\left(\frac{j}{2^{n-1}}\right)\\\notag
&+2^{6n-6}  (-14)E_6\left(\frac{j}{2^{n-1}}\right)+2^{5n-5}56E_5\left(\frac{j}{2^{n-1}}\right)\\\notag
&+2^{4n-4}49E_4\left(\frac{j}{2^{n-1}}\right)+2^{3n-3}(-196)E_3\left(\frac{j}{2^{n-1}}\right)\\\notag
&+2^{2n-2}(- 36)E_2\left(\frac{j}{2^{n-1}}\right)+2^{n-1} 144E_1\left(\frac{j}{2^{n-1}}\right)  \biggr)\frac{1}{\sin[\frac{(2j-1)\pi}{2^n}]}\\\notag
S(11,n)&=\frac{4}{14175}\sum_{j=1}^{2^{n-2}}\biggl(2^{10n-10} (-1) E_{10}\left(\frac{j}{2^{n-1}}\right)+2^{9n-9}5E_9\left(\frac{j}{2^{n-1}}\right)\\\notag
&+2^{8n-8}30  E_8\left(\frac{j}{2^{n-1}}\right)+2^{7n-7}(-150)E_7\left(\frac{j}{2^{n-1}}\right)\\\notag
&+2^{6n-6}  (-273)E_6\left(\frac{j}{2^{n-1}}\right)+2^{5n-5}1365 E_5\left(\frac{j}{2^{n-1}}\right)\\\notag
&+2^{4n-4}820 E_4\left(\frac{j}{2^{n-1}}\right)+2^{3n-3}(-4100)E_3\left(\frac{j}{2^{n-1}}\right)\\\notag
&+2^{2n-2}(- 576)E_2\left(\frac{j}{2^{n-1}}\right)+2^{n-1} 2880 E_1\left(\frac{j}{2^{n-1}}\right)  \biggr)\frac{1}{\sin[\frac{(2j-1)\pi}{2^n}]}\\\notag
S(13,n)&=\frac{4}{467775}\sum_{j=1}^{2^{n-2}}\biggl(2^{12n-12} E_{12}\left(\frac{j}{2^{n-1}}\right)+2^{11n-11}(-6)E_{11}\left(\frac{j}{2^{n-1}}\right)\\\notag
&+2^{10n-10} (-55) E_{10}\left(\frac{j}{2^{n-1}}\right)+2^{9n-9}330E_9\left(\frac{j}{2^{n-1}}\right)\\\notag
&+2^{8n-8}1023  E_8\left(\frac{j}{2^{n-1}}\right)+2^{7n-7}(-6138)E_7\left(\frac{j}{2^{n-1}}\right)\\\notag
&+2^{6n-6}  (-7645)E_6\left(\frac{j}{2^{n-1}}\right)+2^{5n-5}45870 E_5\left(\frac{j}{2^{n-1}}\right)\\\notag
&+2^{4n-4}21076 E_4\left(\frac{j}{2^{n-1}}\right)+2^{3n-3}(-126456)E_3\left(\frac{j}{2^{n-1}}\right)\\\notag
&+2^{2n-2}(- 14400)E_2\left(\frac{j}{2^{n-1}}\right)+2^{n-1} 86400 E_1\left(\frac{j}{2^{n-1}}\right)  \biggr)\frac{1}{\sin[\frac{(2j-1)\pi}{2^n}]}\\\notag
\end{align}
The general form looks like
\begin{align}
S(2m+1,n)&=\frac{2^{2m}}{(2m)!}\sum_{j=1}^{2^{n-2}}\biggl( (-1)^m 2^{(n-1)2m }E_{2m}(\frac{j}{2^{n-1}})+\\
&+(-1)^{m+1}2^{(n-1)(2m-1)}j E_{2m-1}(\frac{j}{2^{n-1}})+\\\notag
&+(-1)^{m+1}2^{(n-1)(2m-2)}\frac{1}{6}(m-1)j(2m-1)E_{2m-2}(\frac{j}{2^{n-1}})+\\\notag
&+(-1)^{m}2^{(n-1)(2m-3)}\frac{1}{6}(m-1)j^2(2m-1)E_{2m-3}(\frac{j}{2^{n-1}})+\\\notag
&+(-1)^{m}2^{(n-1)(2m-4)}\frac{1}{360}m(m-1)(m-2)(2m-1)(2m-3)(5m+1)E_{2m-4}(\frac{j}{2^{n-1}})+\\\notag
&+(-1)^{m+1}2^{(n-1)(2m-5)}\frac{1}{360}m^2(m-1)(m-2)(2m-1)(2m-3)(5m+1)E_{2m-5}(\frac{j}{2^{n-1}})+\\\notag
&+(-1)^{m+1}2^{(n-1)(2m-6)}\frac{1}{45360}m(m-1)(m-2)(m-3)(2m-1)(2m-3)(2m-5)\\\notag
&(35m^2+21m+4)E_{2m-6}(\frac{j}{2^{n-1}})+\\\notag
&+(-1)^{m}2^{(n-1)(2m-7)}\frac{1}{45360}j^2(m-1)(m-2)(m-3)(2m-1)(2m-3)(2m-5)\\\notag
&(35m^2+21m+4)E_{2m-7}(\frac{j}{2^{n-1}})+\\\notag
&+(-1)^{m}2^{(n-1)(2m-8)}\frac{1}{5443200}m(m-1)(m-2)(m-3)(m-4)(2m-1)(2m-3)(2m-5)\\\notag
&(2m-7)(5m+2)(35m^2+28m+9)E_{2m-8}(\frac{j}{2^{n-1}})+\\\notag
&+(-1)^{m+1}2^{(n-1)(2m-9)}\frac{1}{5443200}m^2(m-1)(m-2)(m-3)(m-4)(2m-1)(2m-3)\\\notag
&(2m-5)(2m-7)(5m+2)(35m^2+28m+9)E_{2m-9}(\frac{j}{2^{n-1}})+\\
+\ldots\biggr)\frac{1}{\sin[\frac{(2j-1)\pi}{2^n}]}\\
\end{align}

\begin{theorem}\label{last}
Let $E_j(x)$ be the j-th Euler polynomial and $j\in\mathbb{N}$ then
\begin{align}
S(2m+1,n)&=(-1)^m\frac{2^{2mn}}{(2m)!}\sum_{j=1}^{2^{n-2}}\sum_{k=0}^{m-1}\Biggl(\frac{1}{2^{(n-1)(2k+1)}}\binom{2m-1}{2k}\\\notag
&\Biggl(B_{2k}^{(2m)}(m)\Biggr)\Biggl(2^{n-1}E_{2m-2k}\left(\frac{j}{2^{n-1}}\right)-mE_{2m-2k-1}\left(\frac{j}{2^{n-1}}\right)\Biggr)\Biggr)\frac{1}{\sin[\frac{(2j-1)\pi}{2^n}]}\\\notag
\end{align}
Equivalently,
\begin{align}
S(2m+1,n)&=(-1)^m\frac{1}{(2m)!}\sum_{j=1}^{2^{n-2}}\Biggl((j-m)B_{2m-1}^{(2m)}(m+j)+2^{2mn}\sum_{k=0}^{m-1}E_{2m-2k-1}(0)\frac{1}{2^{(n-1)(2k+1)}}\\\notag
&\Biggl(\binom{2m-1}{2k+1}B_{2k+1}^{(2m)}(m+j)+(j-m)\binom{2m-1}{2k}B_{2k}^{(2m)}(m+j)\Biggr)\Biggr)\frac{1}{\sin[\frac{(2j-1)\pi}{2^n}]}\\\notag
\end{align}
\end{theorem}
\begin{proof}
We know from Prop. \ref{commute}
 $$S(2m+1,n)=2^{2m} \sum_{j=1}^{2^{n-2}}M_n(1,j)\frac{1}{\sin(\frac{(2j-1)\pi}{2^n})}$$
where $M_n$ is the same matrix appearing in Prop. \ref{Mmatrix}.
\end{proof}
\begin{corollary}
Let $E_j(x)$ be the j-th Euler polynomial and $j\in\mathbb{N}$ then
\begin{align}
\zeta(2j+1)&=(-1)^j\frac{2^{2j}\pi^{2j+1}}{(2j)!(2^{2j+1}-1)}\int_{0}^{1/2}E_{2j}(x)\csc (\pi  x)dx\\\notag
\end{align}
\end{corollary}
\begin{proof}
Note, to use these formulas for zeta at odd integers to produce formulas involving integrals, we need to write them in the form $$\lim\sum \Delta x_i f(x_i)$$ For $S(m,n)$ this amounts to letting $\Delta x=\frac{1}{2^{n-1}}$, and distributing the $\frac{1}{2^{(m-1)(n-1)}}$ across the terms. The variable effectively becomes $\frac{j}{2^{n-1}}$.\\ As shown in the general form given before Theorem \ref{last}, when the sum is written in the form of an approximation to an integral, only the high order terms will remain when passing to the limit, those being $E_{2j}(x)$.
We have,
\begin{align}
\zeta(m)&=\lim_{n\to\infty}\left(\frac{\pi^m}{2^{m}-1}\right)\left(\frac{1}{2^{n-1}}\right) \left(\frac{1}{2^{(m-1)(n-1)}}\right)S(m,n) \\\notag
\implies &\\\notag
\zeta(2m+1)&=\lim_{n\to\infty}\left(\frac{\pi^{2m+1}}{2^{2m+1}-1}\right)\left(\frac{1}{2^{n-1}}\right) \left(\frac{1}{2^{(2m)(n-1)}}\right)S(2m+1,n) \\\notag
&=\lim_{n\to\infty}\left(\frac{\pi^{2m+1}}{2^{2m+1}-1}\right)\left(\frac{1}{2^{n-1}}\right) \left(\frac{1}{2^{(2m)(n-1)}}\right)(-1)^m\frac{2^{2mn}}{(2m)!}\sum_{j=1}^{2^{n-2}}\sum_{k=0}^{m-1}\Biggl(2^{(n-1)(-2k-1)}\binom{2m-1}{2k}\\\notag
&\Biggl(B_{2k}^{(2m)}(m)\Biggr)\Biggl(2^{n-1}E_{2m-2k}\left(\frac{j}{2^{n-1}}\right)-mE_{2m-2k-1}\left(\frac{j}{2^{n-1}}\right)\Biggr)\Biggr)\frac{1}{\sin[\frac{(2j-1)\pi}{2^n}]}\\\notag
\zeta(2m+1)&=\lim_{n\to\infty}\left(\frac{\pi^{2m+1}}{2^{2m+1}-1}\right)\left(\frac{1}{2^{n-1}}\right) \left(\frac{1}{2^{(2m)(n-1)}}\right)S(2m+1,n) \\\notag
&=\lim_{n\to\infty}\left(\frac{\pi^{2m+1}}{2^{2m+1}-1}\right)\Delta x_n\left(\frac{1}{2^{(2m)(n-1)}}\right)(-1)^m\frac{2^{2mn}}{(2m)!}\sum_{j=1}^{2^{n-2}}\sum_{k=0}^{m-1}\Biggl(2^{(n-1)(-2k-1)}\binom{2m-1}{2k}\\\notag
&\Biggl(B_{2k}^{(2m)}(m)\Biggr)\Biggl(2^{n-1}E_{2m-2k}\left(x_j\right)-mE_{2m-2k-1}\left(x_j\right)\Biggr)\Biggr)\frac{1}{\sin[(x_j-\frac{1}{2^n})\pi]}\\\notag
\end{align}
In the limit as $n\to\infty$ the only term in the sum over $k$ that will survive is the $k=0$ term. $B_0^{(2m)}(m)=1$.
\begin{align}
\zeta(2m+1)&=\lim_{n\to\infty}\left(\frac{\pi^{2m+1}}{2^{2m+1}-1}\right)\Delta x_n(-1)^m\frac{2^{2m}}{(2m)!}\sum_{j=1}^{2^{n-2}}E_{2m}(x_j)\frac{1}{\sin[x_j\pi]}\\\notag
&=(-1)^m\frac{2^{2m}\pi^{2m+1}}{(2m)!(2^{2m+1}-1)}\int_{0}^{1/2}E_{2m}(x)\csc (\pi  x)dx\\\notag
\end{align}
\end{proof}
\section{Sums Of Negative Even Powers}

As previously noted, if we let
\begin{align}\notag
S(s,n)&=\sum_{j=1}^{2^{n-2}}\frac{1}{\sin^s\left(\frac{(2j-1)\pi}{2^n}\right)}\\\notag
\end{align}
Then
\begin{align}\notag
S(2,n)&=\frac{1}{2}2^{2n-2}\\\notag
S(4,n)&=\frac{1}{6}\left(2^{4n-4}+2(2^{2n-2})\right)\\\notag
S(6,n)&=\frac{1}{30}(2(2^{6n-6})+5(2^{4n-4})+8(2^{2n-2}))\\\notag
S(8,n)&=\frac{1}{630}\left(17(2^{8n-8})+56(2^{6n-6})+98(2^{4n-4})+144(2^{2n-2})\right)\\\notag
\end{align}
We can see a pattern if we write these as
\begin{align}\notag
S(2,n)&=\frac{1}{2}2^{2n-2}\\\notag
S(4,n)&=\frac{1}{2}\frac{1}{3}2^{4n-4}+\frac{2}{6}2^{2n-2}\\\notag
S(6,n)&=\frac{1}{2}\frac{2}{15}2^{6n-6}+\frac{3}{6}\frac{1}{3}2^{4n-4}+\frac{4}{15}2^{2n-2}\\\notag
S(8,n)&=\frac{1}{2}\frac{17}{315}2^{8n-8}+\frac{4}{6}\frac{2}{15}2^{6n-6}+\frac{7}{15}\frac{1}{3}2^{4n-4}+\frac{8}{35}2^{2n-2}\\\notag
S(10,n)&=\frac{1}{2}\frac{62}{2835}2^{10n-10}+\frac{5}{6}\frac{17}{315}2^{8n-8}+\frac{13}{18}\frac{2}{15}2^{6n-6}+\frac{82}{189}\frac{1}{3}2^{4n-4}+\frac{64}{315}2^{2n-2}\\\notag
S(12,n)&=\frac{1}{2}\frac{1382}{155925}2^{12n-12}+\frac{6}{6}\frac{62}{2835}2^{10n-10}+\frac{31}{30}\frac{17}{315}2^{8n-8}+\frac{139}{189}\frac{2}{15}2^{6n-6}+\frac{1916}{4725}\frac{1}{3}2^{4n-4}+\frac{128}{693}2^{2n-2}\\\notag
\end{align}
We can recognize the terms in columns as corresponding to the coefficients of the Taylor series expanson of certain functions.
\begin{align}\notag
S(2 k,n)&=\sum_{j=1}^k(-1)^{k+j}\frac{1}{2(2(k-j))!)}\Bigl(\lim_{t\to 0}\frac{  \mathrm{  d^{2(k-j)}  }   }{\mathrm{d}t^{2(k-j)}  }        \Bigl(\frac{t}{\sinh(t)}\Bigr)^x\vert_{x=2 k}\Bigr)\Bigl(\frac{1}{(2j)!}\frac{\mathrm{  d^{2j}  }  }{\mathrm{d}t^{2j}}  t\tan(t)\vert_{t=0}\Bigr)2^{2 j (n-1)}\\\notag
\end{align}
This looks like the terms in the $2k$-th derivative of a product. We can construct such a product:
\begin{align}
S(2k,n)&=\lim_{t\to 0}\frac{  \mathrm{  d^{2k}  }   }{\mathrm{d}t^{2k}  }  \Bigl( \frac{(-1)^{k+1}}{2(2k)!}     \Bigl(\frac{ t}{\sinh(t/2^{n-1})}\Bigr)^{2k} t\tanh(t)\Bigr)\\\notag
\end{align}

\begin{theorem}\label{theorem1}
 Let $B(j)$ be the $j$-th Bernoulli number and let $B_n^{(\alpha)}(x)$ be the generalized Bernoulli polynomial of degree $n$ in $x$.
\begin{align}\notag
S(2 k,n)&=\sum_{j=1}^k(-1)^{k+1}\frac{2^{2j(n-1)+2k-1}(2^{2j}-1)}{(2(k-j))!(2j)!}B_{2(k-j)}^{(2k)}(k) B(2 j)\\\notag
&=\sum_{j=1}^k(-1)^{j+k}\frac{2^{2j(n-2)+2k}(2^{2j}-1)}{\pi^{2j}(2(k-j))!}B_{2(k-j)}^{(2k)}(k) \zeta(2 j)\\\
\end{align}
Equivalently,
\begin{align}
S(2 k,n)&=(-1)^{k+1}\frac{2^{2k-1}}{(2k)!}\sum_{j=1}^k2^{2jn}\frac{2^{2j}-1}{2^{2j}}\binom{2k}{2j}B_{2(k-j)}^{(2k)}(k) B(2 j)\\\label{sumsine}
\end{align}
\end{theorem}
\begin{proof}
See Prop. \ref{lastprop} and the comments preceding the proposition.
\end{proof}
Here's what this looks like.
\begin{align}\notag
S(2,n)&=\frac{3}{\pi^2}\zeta(2)2^{2n-2}\\\notag
S(4,n)&=\frac{15}{\pi^4}\zeta(4)2^{4n-4}+\frac{2}{\pi^2}\zeta(2)2^{2n-2}\\\notag
S(6,n)&=\frac{63}{\pi^6}\zeta(6)2^{6n-6}+\frac{15}{\pi^4}\zeta(4)2^{4n-4}+\frac{8}{5\pi^2}\zeta(2)2^{2n-2}\\\notag
S(8,n)&=\frac{255}{\pi^8}\zeta(8)2^{8n-8}+\frac{84}{\pi^6}\zeta(6)2^{6n-6}+\frac{14}{\pi^4}\zeta(4)2^{4n-4}+\frac{48}{35\pi^2}\zeta(2)2^{2n-2}\\\notag
S(10,n)&=\frac{1023}{\pi^{10}}\zeta(10)2^{10n-10}+\frac{425}{\pi^8}\zeta(8)2^{8n-8}+\frac{91}{\pi^6}\zeta(6)2^{6n-6}+\frac{820}{63\pi^4}\zeta(4)2^{4n-4}+\frac{128}{105\pi^2}\zeta(2)2^{2n-2}\\\notag
S(12,n)&=\frac{4095}{\pi^{12}}\zeta(12)2^{12n-12}+\frac{2046}{\pi^{10}}\zeta(10)2^{10n-10}+\frac{527}{\pi^8}\zeta(8)2^{8n-8}+\frac{278}{3\pi^6}\zeta(6)2^{6n-6}+\frac{3832}{315\pi^4}\zeta(4)2^{4n-4}+\\\notag
&+\frac{256}{231\pi^2}\zeta(2)2^{2n-2}\\\notag
\end{align}
\begin{align}\notag
S(2k,n)&=\frac{2^{2k}-1}{\pi^{2k}}\zeta(2k)2^{2k(n-1)}+{\rm lower \ order\  powers\ of\  2}\\\notag
\end{align}
Note that computing $S(6,1000)$ from the definition involves performing $2^{998}$ additions, whereas using this formula involves 3 additions.
We know the following:
\begin{proposition}
Sum over odd values:\\
Assume $m\in\mathbb{Z}$, $m>2$ then
$$\zeta(m)=\lim_{n\to\infty}\left(\frac{2^m\pi^m}{2^{m}-1}\right) \sum _{i=1}^{2^{n-2}}\frac{1}{ \left(2^n\sin\left(\frac{\left(2 i-1\right)\pi}{2^n}\right)\right)^m} $$\notag\\
\end{proposition}
\begin{proof}
See ref. [1],[2].
\end{proof}
Using the above notation, 
\begin{align}\notag
\zeta(2k)&=\lim_{n\to\infty}\left(\frac{2^{2k}\pi^{2k}}{2^{2k}-1}\right)\frac{1}{2^{2kn}} S(2k,n)\\\notag
&=\lim_{n\to\infty}\left(\frac{2^{2k}\pi^{2k}}{2^{2k}-1}\right)\frac{1}{2^{2kn}}  \sum_{j=1}^k\frac{(-1)^{j+k}}{(2(k-j))!)}\Bigl(\lim_{t\to 0}\frac{  \mathrm{  d^{2(k-j)}  }   }{\mathrm{d}t^{2(k-j)}  }        \Bigl(\frac{t}{\sinh(t)}\Bigr)^x\vert_{x=2 k}\Bigr)\Bigl(\frac{2^{2j}-1}{\pi^{2 j}}\zeta(2 j)\Bigr)2^{2 j (n-1)}\\\notag
&=\lim_{n\to\infty}\left(\frac{2^{2k}\pi^{2k}}{2^{2k}-1}\right)\frac{1}{2^{2kn}}\sum_{j=1}^k(-1)^{j+k}\frac{2^{2j(n-2)+2k}(2^{2j}-1)}{\pi^{2j}(2(k-j))!}B_{2(k-j)}^{(2k)}(k) \zeta(2 j)\\\notag
\end{align}
The $k$-th term in  the sum is
\begin{align}\notag
&\left(\frac{2^{2k}\pi^{2k}}{2^{2k}-1}\right) \frac{1}{2^{2kn}} \frac{2^{2k(n-2)+2k}(2^{2k}-1)}{\pi^{2k}(2(k-k))!}B_{2(k-k)}^{(2k)}(k) \zeta(2 k)\\\notag
&= \zeta(2 k)\\\notag
\end{align}

If we let
$$Z(m,n)=\left(\frac{2^m\pi^m}{2^{m}-1}\right) \sum _{i=1}^{2^{n-2}}\frac{1}{ \left(2^n\sin\left(\frac{\left(2 i-1\right)\pi}{2^n}\right)\right)^m} $$\notag\\
then
$$Z(2k,n)=\zeta(2k)+R(2k,n)$$
where the formula for the error term $R(2k,n)$ is
\begin{align}\notag
R(2k,n))&=\left(\frac{2^{2k}\pi^{2k}}{2^{2k}-1}\right)\sum_{j=1}^{k-1}(-1)^{j+k}\frac{2^{2j(n-2)+2k}(2^{2j}-1)}{\pi^{2j}(2(k-j))!}B_{2(k-j)}^{(2k)}(k) \zeta(2 j)\\\notag
\end{align}
In Section 3 we saw an example of a matrix equation associated with negative odd powers of sine:

\begin{equation}\notag
\left(
\begin{array}{c}
\sin^{-3}\left(\frac{\pi}{16}\right)   \\
\sin^{-3}\left(\frac{3\pi}{16}\right)  \\
\sin^{-3}\left(\frac{5\pi}{16}\right)  \\
\sin^{-3}\left(\frac{7\pi}{16}\right)  \\
\end{array}
\right)
=8\left(
\begin{array}{cccc}
 2 & 5 & 7 & 8 \\
 7 & 2 & -8 & 5 \\
 5 & 8 & 2 & -7 \\
 -8 & 7 & -5 & 2 \\
\end{array}
\right)
\left(
\begin{array}{c}
\sin\left(\frac{\pi}{16}\right)   \\
\sin\left(\frac{3\pi}{16}\right)  \\
\sin\left(\frac{5\pi}{16}\right)  \\
\sin\left(\frac{7\pi}{16}\right)  \\
\end{array}
\right)
\end{equation}

We can also produce analogous equations for negative even powers:
\begin{equation}\notag
\left(
\begin{array}{c}
\sin^{-4}\left(\frac{\pi}{16}\right)   \\
\sin^{-4}\left(\frac{3\pi}{16}\right)  \\
\sin^{-4}\left(\frac{5\pi}{16}\right)  \\
\sin^{-4}\left(\frac{7\pi}{16}\right)  \\
\end{array}
\right)
=16\left(
\begin{array}{cccc}
 8 & 15 & 20 & 11 \\
 20 & -15 & -8 & 11 \\
 -20& -15 & 8 & 11 \\
 -8 & 15 & -20 & 11 \\
\end{array}
\right)
\left(
\begin{array}{c}
\sin\left(\frac{\pi}{8}\right)   \\
\sin\left(\frac{\pi}{4}\right)  \\
\sin\left(\frac{3\pi}{8}\right)  \\
\sin\left(\frac{\pi}{2}\right)  \\
\end{array}
\right)
\end{equation}
The matricies for negative even powers do not have the nice commutative property that the negative odd power matricies have. However, when you sum the values in one of the first three columns, they total to zero. This implies that the sum of those negative powers on the left of the equals sign corresponds to the sum of the values in the last column (multiplied by 16), which is just 4 times the element in the first row. This is true in general for these negative even power matricies. 
\par
\begin{lemma}

Assume $m\in\mathbb{N}$, $m>0$. Regarding the $2^{n-2}x2^{n-2}$ matrices, $M_n$ which satisfy $$\frac{1}{(\sin(\frac{(2j-1)\pi}{2^n}))^{2m}}=2^{2m} \sum_{k=1}^{2^{n-2}}M_n(j,k)\sin(\frac{k\pi}{2^{n-1}})$$
The elements in columns $1,...,2^{n-2}-1$ sum to zero.
\end{lemma}
\begin{proof}
For a given power $2m$, all the matricies for a fixed value of $n$, rows $2$ through $2^{n-2}$ contain the same elements as the first row, but have order permuted and the sign may change. But the permutastions and sign are the same for all $m$ at that fixed $n$. This is because the isomorphism that sends $\sin(\frac{(2j-1)\pi}{2^n})$ to $\sin(\frac{\pi}{2^n})$ determines the permutation (and signs) of the first row into the $j$-th row. Writing out the matrix for the case $2m=2$, we see that every entry $M(i,k)$ in column $k$,$1\le k\le 2^{2n-2}-1$, is accompanied by an entry $M(i,j)$ where $M(i,k)=-M(i,j)$ and thus the sum is zero. So, we may conclude this is the case for all $m>0$. To prove this for all $m$ and $n$, we only need to prove it for all $n$ with $m=2$. We know that
$$\sum_{i=1}^{2^{n-2}}\left(\frac{1}{  \sin\left(\frac{\left(2 i-1\right)\pi}{2^n}\right)   }\right)^2=2^{2n-3}$$ (see ref. [1]). We will see below  that the entry $M(1,2^{n-2})$ equals the entry $\bar M(1,1)$ where $\bar{M}$ is the matrix for (negative) power $2m+1$, $n$. Referring to Theorem \ref{last}, substituting $m=1$ and evaluating we get the desired value. Note the sum $S(2,n)=2^{n-2}M(1,2^{n-2})$ since we sum all rows of $M$. This tells us the contribution of the other columns is zero. It doesn't quite tell us about individual columns.
\end{proof}
We can pass from the first row of a matrix corresponding to $\frac{1}{\sin\left(\frac{\pi}{2^{n}}\right)^{2m+1} }$ to the first row of the matrix corresponding to $\frac{1}{\sin\left(\frac{\pi}{2^{n}}\right)^{2m+2} }$ by multiplying by $\frac{1}{\sin\left(\frac{\pi}{2^{n}}\right) }$. The following lemma tells us what the result is.

\begin{lemma}\label{oddtoeven}
\begin{align}
\frac{1}{\sin\left(\frac{\pi}{2^{n}}\right) }\sum_{i=1}^{2^{n-2}}c_i\sin\left(\frac{(2i-1)\pi}{2^{n}}\right)&=\sum_{j=1}^{2^{n-2}-1}\biggl(\sum_{i=2^{n-2}-j+1}^{2^{n-2}}2c_i \biggr)\sin\left(\frac{j\pi}{2^{n-1}}\right)+\sum_{i=1}^{2^{n-2}}c_i\\\notag
\end{align}
\end{lemma}
\begin{proof}
This procedes just as in the proof of Lemma \ref{sinequotient}.
\begin{align}
\frac{1}{\sin\left(\frac{\pi}{2^n}\right) }\sin\left(\frac{(2i-1)\pi}{2^n}\right)&=\sum_{k=1}^{2^{n-2}}\sin\left(\frac{2(i-k)\pi}{2^n}\right)+\sum_{k=1}^{2^{n-2}}\sin\left(\frac{2(i+k-1)\pi}{2^n}\right)\\\notag
\end{align}
So,
\begin{align}
\sum_{i=1}^{2^{n-2}}c_i\frac{1}{\sin\left(\frac{\pi}{2^n}\right) }&\sin\left(\frac{(2i-1)\pi}{2^n}\right)=\sum_{i=1}^{2^{n-2}}c_i\left(\sum_{k=1}^{2^{n-2}}(\sin\left(\frac{2(i-k)\pi}{2^n}\right)+\sum_{k=0}^{2^{n-2}-1}\sin\left(\frac{2(i+k)\pi}{2^n}\right)\right)\\\notag
&=\sum_{i=1}^{2^{n-2}}c_i\left(\sum_{j=i-2^{n-2}}^{i-1}(\sin\left(\frac{2j\pi}{2^n}\right)+\sum_{j=i}^{i+2^{n-2}-1}\sin\left(\frac{2j\pi}{2^n}\right)\right)\\\notag
&=\sum_{j=1-2^{n-2}}^{2^{n-2}-1}  \sum_{i=\max(1,j+1)}^{\min(2^{n-2},2^{n-2}+j)}c_i\sin\left(\frac{2j\pi}{2^n}\right)+\sum_{j=1}^{2^{n-1}-1}\sum_{i=\max(1,j-2^{n-2}+1)}^{\min(2^{n-2}-1,j)}c_i\sin\left(\frac{2j\pi}{2^n}\right)\\\notag
&=\sum_{j=1-2^{n-2}}^{0}  \sum_{i=\max(1,j+1)}^{\min(2^{n-2},2^{n-2}+j)}c_i\sin\left(\frac{2j\pi}{2^n}\right)\\\notag
&+\sum_{j=1}^{2^{n-2}-1}\left(\sum_{i=\max(1,j+1)}^{\min(2^{n-2},2^{n-2}+j)}c_i+\sum_{i=\max(1,j-2^{n-2}+1)}^{\min(2^{n-2}-1,j)}c_i\right)\sin\left(\frac{2j\pi}{2^n}\right)\\\notag
&+\sum_{j=2^{n-2}}^{2^{n-1}}\sum_{i=\max(1,j-2^{n-2}+1)}^{\min(2^{n-2},j)}c_i\sin\left(\frac{2j\pi}{2^n}\right)\\\notag
&=\sum_{j=1-2^{n-2}}^{0}  \sum_{i=1}^{2^{n-2}+j}c_i\sin\left(\frac{2j\pi}{2^n}\right)+\sum_{j=1}^{2^{n-2}-1}\left(\sum_{i=j+1}^{2^{n-2}}c_i+\sum_{i=1}^{j}c_i\right)\sin\left(\frac{2j\pi}{2^n}\right)\\\notag
&+\sum_{j=2^{n-2}}^{2^{n-1}}\sum_{i=\max(1,j-2^{n-2}+1)}^{2^{n-2}}c_i\sin\left(\frac{2j\pi}{2^n}\right)\\\notag
&=\sum_{j=1}^{2^{n-2}-1}  \sum_{i=1}^{2^{n-2}-j}-c_i\sin\left(\frac{j\pi}{2^{n-1}}\right)+\sum_{j=1}^{2^{n-2}-1}\sum_{i=1}^{2^{n-2}}c_i\sin\left(\frac{j\pi}{2^{n-1}}\right)\\\notag
&+\sum_{j=1}^{2^{n-2}}\sum_{i=\max(1,j-2^{n-2}+1)}^{2^{n-2}}c_i\sin\left(\frac{j\pi}{2^{n-1}}\right)\\\notag
&=\sum_{j=1}^{2^{n-2}-1}\biggl(\sum_{i=2^{n-2}-j+1}^{2^{n-2}}2c_i \biggr)\sin\left(\frac{j\pi}{2^{n-1}}\right)+\sum_{i=1}^{2^{n-2}}c_i\\\notag
\end{align}
\end{proof}
If, as in Prop. \ref{negodd}, we assume $r\in\mathbb{Z}$, $r$ odd, and consider the $2^{n-2}x2^{n-2}$ matrices, $M_n$ which satisfy $$\frac{1}{(\sin(\frac{(2j-1)\pi}{2^n}))^r}=2^r \sum_{k=1}^{2^{n-2}}M_n(j,k)\sin(\frac{(2k-1)\pi}{2^n})$$ Then setting $c_i=2^rM_n(1,i)$, Lemma \ref{oddtoeven} gives us the matrix $\bar{M}_n$ for power $r+1$. We know that $\sum_{i=1}^{2^{n-2}}\frac{1}{\sin\left(\frac{(2i-1)\pi}{2^n}\right)^{r+1}}=2^{n-2}\bar{M}_n(1,2^{n-2})$. So, we now have an equivalent formulation of Theorem \ref{theorem1}:
\begin{proposition}
Let $B(j)$ be the $j$-th Bernoulli number, $B_n^{(\alpha)}(x)$ be the generalized Bernoulli polynomial of degree $n$ in $x$, $E_j(x)$ be the j-th Euler polynomial.
\begin{align}\notag
\frac{2^{2m+1}}{(2m+2)!}&\sum_{j=1}^{m+1}2^{2jn}\frac{2^{2j}-1}{2^{2j}}\binom{2m+2}{2j}B_{2(m+1-j)}^{(2m+2)}(m+1) B(2 j)\\\notag
=&2^{n-1}\frac{2^{mn}}{(2m)!}\sum_{j=1}^{2^{n-2}}\sum_{k=0}^{m-1}\Biggl(\frac{1}{2^{(n-1)(2k+1)}}\binom{2m-1}{2k}\\\notag
&\Biggl(B_{2k}^{(2m)}(m)\Biggr)\Biggl(2^{n-1}E_{2m-2k}\left(\frac{j}{2^{n-1}}\right)-mE_{2m-2k-1}\left(\frac{j}{2^{n-1}}\right)\Biggr)\Biggr)\\\notag
\end{align}
\end{proposition}
\begin{proof}
The proof of this identity will follow from the following, simpler identity, in Proposition \ref{lastprop}.
\end{proof}

If  $$\frac{1}{(\sin(\frac{\pi}{2^n}))^r}=2^r \sum_{k=1}^{2^{n-2}}M_n(1,k)\sin(\frac{(2k-1)\pi}{2^n})$$ then 
\begin{align}
\frac{1}{(\sin(\frac{\pi}{2^n}))^{r-1}}&=2^r \sum_{k=1}^{2^{n-2}}M_n(1,k)\sin(\frac{(2k-1)\pi}{2^n})\sin(\frac{\pi}{2^n})\\\notag
&=2^r \sum_{k=1}^{2^{n-2}}M_n(1,k)\frac{1}{2}\left(\cos(\frac{(2k-2)\pi}{2^n})-\cos(\frac{(2k)\pi}{2^n})\right)\\\notag
\end{align}
We are interested in the coefficient of $\cos(0)$, which is clearly $2^{r-1} M_n(1,1)$. This translates into the following equality.
\begin{proposition}\label{lastprop}
Let $B(j)$ be the $j$-th Bernoulli number, $B_n^{(\alpha)}(x)$ be the generalized Bernoulli polynomial of degree $n$ in $x$, $E_j(x)$ be the j-th Euler polynomial.
\begin{align}\notag
-\frac{2^{2m-1}}{2^{n-2}(2m)!}&\sum_{j=1}^{m}2^{2jn}\frac{2^{2j}-1}{2^{2j}}\binom{2m}{2j}B_{2(m-j)}^{(2m)}(m) B(2 j)\\\notag
=&\frac{2^{mn}}{(2m)!}\sum_{k=0}^{m-1}\Biggl(\frac{1}{2^{(n-1)(2k+1)}}\binom{2m-1}{2k}\\\notag
&\Biggl(B_{2k}^{(2m)}(m)\Biggr)\Biggl(2^{n-1}E_{2m-2k}\left(\frac{1}{2^{n-1}}\right)-mE_{2m-2k-1}\left(\frac{1}{2^{n-1}}\right)\Biggr)\Biggr)\\\notag
\end{align}
\end{proposition}
\begin{proof}
This is Lemma \ref{oddeven}.
\end{proof}
\begin{section}{{\textbf{Bibliography}}}
\frenchspacing
\begin{itemize}
\item{[1]} L. Fairbanks, ``Notes On An Approach to Apery's Constant',  arXiv:2206.11256 [math.NT]
\item{[2]} L. Fairbanks, ``Powers of Cosine and Sine, arXiv:2308.04437  [math.NT]
\item{[3]} S. Gaboury, R. Tremblay, B. J. Fugere, ``Some Explicit Formulasd For Certain New Classes Of Bernoulli, Euler And Genocchi Polynomials", Procedings of the Jangjeon Mathematical mSociety, Jan. 2014
\item{[4]} Goldberg, K.; Newman, M; Haynsworth, E. (1972), "Stirling Numbers of the First Kind, Stirling Numbers of the Second Kind", in Abramowitz, Milton; Stegun, Irene A. (eds.), Handbook of Mathematical Functions with Formulas, Graphs, and Mathematical Tables, 10th printing, New York: Dover, pp. 824–825
\item{[5]} H.M. Srivastava, A. Pinter.,"Remarks on Some  Relationships Between the Bernoulli and Euler Polynomials", Applied Mathematics Letters 17 (2004) 375-380
\end{itemize}
\end{section}
\end{flushleft}
\end{document}